\documentclass[10pt,a4paper,fleqn]{article}

\usepackage[paper=a4paper,left=25mm,right=25mm,top=25mm,bottom=25mm]{geometry}

\usepackage[utf8]{inputenc}
\usepackage[english]{babel}
\usepackage[T1]{fontenc}
\usepackage{lmodern}
\usepackage{amsmath, amstext, amsfonts, mathrsfs}
\usepackage{amsfonts}
\usepackage{amssymb}
\usepackage{booktabs}
\usepackage{makecell}
\usepackage{dsfont}
\usepackage{mathrsfs}
\usepackage{graphicx}
\usepackage{epsfig}
\usepackage[numbers]{natbib}
\usepackage{color}
\usepackage{bm}
\usepackage{verbatim}
\usepackage{stmaryrd}
\usepackage{appendix}
\usepackage[colorlinks=false,linkcolor=darkblue]{hyperref}
\usepackage{bbm}
\usepackage{enumitem}
\usepackage{nicefrac}
\usepackage{xfrac}
\usepackage{units}
\usepackage{mathtools}
\usepackage{url}
\usepackage{tikz}
\usetikzlibrary{decorations.markings, arrows.meta}
\usepackage{cancel}
\usepackage[labelfont={bf,sf},font={small},  labelsep=space]{caption}
\usepackage{acronym}
\usepackage{mathrsfs}
\usepackage{fancyhdr}
\usepackage[T1]{fontenc}
\usepackage{caption, booktabs}\usepackage{textcomp}
\usepackage{multirow}
\usepackage{tikz}
\usetikzlibrary{arrows.meta}
\usetikzlibrary{positioning,arrows}
\usetikzlibrary{matrix}
\usepackage{hologo}
\usepackage{xcolor}
\usepackage{forest}
\usetikzlibrary{matrix,fit}
\tikzset{  mleftdelimiter/.style={inner ysep=0pt, inner xsep=1ex,left delimiter=\{,label={[label distance=3mm]left:#1}}
}

\usepackage[Bjarne]{fncychap}

\usepackage{setspace}
\usepackage[squaren]{SIunits}

\usepackage{overpic}
\usepackage{subcaption}

\definecolor{light-gray}{gray}{0.95}
\definecolor{darkblue}{rgb}{0,0,.5}

\newcommand{\E}{\mathbb{E}}

\newcommand{\N}{\mathbb{N}}

\newcommand{\R}{\mathbb{R}}

\newcommand{\cC}{\mathcal{C}}

\newcommand{\cF}{\mathcal{F}}

\newcommand{\sR}{\mathsf{R}}

\newcommand{\cS}{\mathcal{S}}
\newcommand{\cU}{\mathcal{U}}

\newcommand{\cT}{\mathcal{T}}
\newcommand{\cW}{\mathcal{W}}
\newcommand{\dW}{\overrightarrow{\mathcal{W}}}

\providecommand{\keywords}[1]{\textbf{Keywords } #1}
\newcommand{\eqd}{\stackrel{\mathrm{d}}=}

\newcommand{\1}{\mathds{1}}

\newcommand{\de}{\mathrm{\,d}}

\newcommand{\Var}{\mathrm{Var}}

\definecolor{light-gray}{gray}{0.95}
\definecolor{darkblue}{rgb}{0,0,.5}
\definecolor{foxred}{rgb}{0.7, 0.11, 0.11}

\newtheorem{theorem}{Theorem}[section]
\newtheorem{proposition}[theorem]{Proposition}
\newtheorem{corollary}[theorem]{Corollary}
\newtheorem{definition}[theorem]{Definition}
\newtheorem{lemma}[theorem]{Lemma}
\newtheorem{example}[theorem]{Example}
\newtheorem{remark}[theorem]{Remark}

\newenvironment{proof}[1][{Proof:}]{\begin{trivlist}
\item[\hskip \labelsep {\bfseries #1}]}{\hfill$\blacksquare$ \end{trivlist}}

\selectfont

\author{  Jonathan Ansari\textsuperscript{1}
  \\
  \textsuperscript{1} University of Salzburg, Austria
}

\title{On a copula product characterizing independence \\and perfect functional dependence}

\begin{document}

\maketitle

\begin{abstract}
Recently studied dependence measures take values in \([0,1]\), where \(0\) characterizes independence of \(X\) and \(Y\), and \(1\) characterizes perfect functional (not necessarily monotone) dependence of \(Y\) on \(X\).
The most prominent example is Chatterjee's rank correlation, which is based on the concept of conditional independence. In contrast, we show that dependence measures such as Wasserstein correlations and rearranged dependence measures are naturally linked to the concept of conditional comonotonicity.
This connection is captured by a new copula product \(\cT(C)=C\vee\Pi\) which models conditional comonotonicity and characterizes independence and perfect functional dependence. As a main contribution, we prove that the mapping \(\cT\) acts as a reflection on the class of stochastically increasing copulas. Further, \(\cT^2=\cT\circ\cT\) projects a copula onto its increasing rearranged copula. To better understand the behavior of such dependence measures, we also study fixed points, ordering results, and continuity properties of \(\cT\).
Our results demonstrate that conditional comonotonicity is a rather intrinsic feature of dependence measures, whereas conditional independence underlying Chatterjee's rank correlation is a quite exceptional property.

\end{abstract}

\maketitle


\section{Introduction and main results}\label{secIntro}

Quantifying dependence between random variables is a classical and important topic that continues to be an active field of research.
Well-known measures of association, such as the Pearson correlation, Kendall's tau or Spearman's rho, quantify the degree of positive or negative (linear) dependence between two random variables \(X\) and \(Y\).
However, they fail to detect non-linear/non-monotone dependence. For example, take \(X\) standard normal and \(Y=X^2\), then all these quantities are zero even though \(Y\) is a deterministic function of \(X\).
Motivated by the seminal papers \cite{siburg2013,chatterjee2020,chatterjee2021},
over the last decade, research has focused on new dependence measures \(\kappa(Y,X)\) that quantify the strength of functional dependence of \(Y\) on \(X\); see
\cite{chatterjee2023} for a recent review.
To be precise, such dependence measures \(\kappa\) are assumed to satisfy the following axioms:

\begin{enumerate}[label=(\Roman*)]
    \item \label{prop1} \(\kappa(Y,X) \in [0,1]\).
    \item \label{prop2}     \(\kappa(Y,X) = 0\) if and only if \(X\) and \(Y\) are independent.
    \item \label{prop3}     \(\kappa(Y,X)=1\) if and only if \(Y\) \emph{perfectly depends} on \(X\), i.e. there exists a measurable (not necessarily monotone) function \(f\) such that \(Y = f(X)\) almost surely.
\end{enumerate}
The certainly most famous and simplest dependence measure satisfying Axioms \ref{prop1}--\ref{prop3} is \emph{Chatterjee's correlation coefficient}. Its population version is defined for real-valued random variables \(Y\) and \(X\) by
\begin{align}\label{eq_chatt_xi}
    \xi(Y,X) := \frac{\int_\R \Var(P(Y\geq y\mid X)) d P^Y(y)}{\int_\R \Var(\1_{Y\geq y}) d P^Y(y)},
\end{align}
see also \cite{Gamboa-Klein-Lagnoux-2018}.
Unlike classical association measures, dependence measures such as \(\xi\) characterize independence and detect arbitrarily complex functional relationships of \(Y\) on \(X\).

In this paper, we focus on the population level of such dependence measures and study dependence structures underlying them.
To keep the setting simple, we consider bivariate random vectors \((Y,X)\) with uniform marginals on \((0,1)\).\footnote{
Extensions of our results to random vectors with arbitrary continuous distribution functions are straightforward. For \((Y',X')\) having a continuous distribution function, consider then the transformed random vectors \((Y,X)=(F_{Y'}(Y'),F_{X'}(X'))\).}
Hence, the distribution function of \((Y,X)\) is a bivariate copula, i.e., a distribution function \(C\colon [0,1]^2\to [0,1]\) that satisfies \(C(u,1) = C(1,u) = u\) for all \(u\in [0,1]\).
We write \(C_{Y,X}\) for the copula of \((Y,X)\).
Besides separating the dependence structure from the marginal distributions via Sklar's theorem (see Equation \eqref{eq_sklar}), working on the copula scale has the advantage of simple geometric visualizations on the unit square (see Figures \ref{fig:S_Gauss} and \ref{fig3}) and a natural link to doubly stochastic measures and matrices. Moreover, the uniform marginals of copulas allow the associated measures to be disintegrated with respect to the Lebesgue measure, which simplifies calculations.
We refer to Section \ref{sec2} for the concepts relevant to this article, and to the textbooks \cite{fdsempi2016,Nelsen-2006} for an introduction to copula theory.

Important copulas that we use throughout are the \emph{independence copula} \(\Pi(u,v) := uv\) and the \emph{upper Fr\'{e}chet copula} \(M(u,v) := \min\{u,v\}\), which model independence and comonotonicity,\footnote{\(X\) and \(Y\) are said to be \emph{comonotone} if there exist a random variable \(Z\) and increasing functions \(f\) and \(g\) such that \(X = f(Z)\) and \(Y= g(Z)\) almost surely. Roughly speaking, in this case, \(X\) and \(Y\) are perfectly \emph{positively} dependent.} respectively.
We call a copula \(C\) as \emph{perfect-dependence copula} if, for \((Y,X)\sim C\), the first component \(Y\) perfectly depends on the second component \(X\). Obviously, \(M\) is a perfect dependence copula (take \(Y=X\)), but it is far from the only one; for example, any shuffle-of-min copula \cite{Mikusinski-1992} is a perfect dependence copula, and, of course, \(C_{Y',X'}\) with \(Y'=f(X')\) whenever \(X'\) and \(Y'\) have a continuous distribution function.

\subsection{Conditional independence and Markov products}
For continuous distribution functions, Chatterjee's \(\xi\) in \eqref{eq_chatt_xi} depends only on the underlying copula \(C \), and it reduces to the \emph{Dette-Siburg-Stoimenov measure}
\begin{align}\label{def_xi_cop}
    \xi(C) =  6 \int_0^1\int_0^1 (\partial_2 C(u,t))^2 \de t \de u - 2;
\end{align}
see \cite{siburg2013}. Here, \(\partial_2 C\) denotes the derivative of \(C\) with respect to its second component. The popularity of Chatterjee’s \(\xi\)
is largely due to the fact that it has an estimator of a remarkably simple form. This estimator is based on nearest-neighbor statistics and exploits the representation of \(\xi\) as a product of conditionally independent random variables. More specifically, at the level of copulas, \(\xi\) can be written as
\begin{align}\label{rep_xi}
    \xi(C) = \psi(C\ast C);
\end{align}
see \cite{Fuchs-2024}. Here,
 \(\psi\) is Spearman's footrule defined by \(\psi(D) = 6 \int_0^1 D(t,t) \de t - 2\) which evaluates a copula \(D\) only on its diagonal; see e.g. \cite{Genest-2010}.
Further, \(D\ast E\) denotes the \emph{Markov product}\footnote{For notational convenience, we use here a variant of the Markov product, which is usually defined by \(\int_0^1 \partial_2 D(u,t) \partial_1 E(t,v) \de t\) reflecting the temporal structure of the underlying Markov process; see \cite{darsow-1992,fdsempi2016}.} or \emph{\(\ast\)-product} of two bivariate copulas \(D\) and \(E\), which we define by
\begin{align}\label{def:stern_product}
    D\ast E(u,v) = \int_0^1 \partial_2 D(u,t) \partial_2 E(v,t) \de t,  \qquad (u,v)\in [0,1]^2.
\end{align}
It is well-known that \(D\ast E\) is itself a copula as a particular case of the copula product in \eqref{eq:gencopprod}.
Note that, for a bivariate random vector \((U,Z)\) with distribution function \(D\), the copula derivative \(\partial_2 D(\cdot,t)\) describes the conditional distribution of \(U\) given \(Z=t\); similarly, for \((V,Z)\sim E\), the copula derivative \(\partial_2 E(\cdot,t)\) specifies the distribution of \(V\) conditional on \(Z=t\).
Hence, since the integrand in \eqref{def:stern_product} has a product form, it models conditional independence of \(U\) and \(V\) given \(Z=t\); see \cite{darsow-1992}. This explains why the Markov product is also referred to as \emph{conditional independence product}.

The representation of \(\xi\) in \eqref{rep_xi} suggests that dependence measures satisfying the Axioms \ref{prop1}--\ref{prop3} are closely related to Markov products. However, this is not the case since, outside of the diagonal, the Markov product \(C\ast C\) may attain values smaller than those of the independence copula; see Eq. \eqref{eq_cop_diagn} below.

\subsection{Conditional comonotonicity and upper products}

As we demonstrate in the present paper, dependence measures with the properties \ref{prop1}--\ref{prop3} are more naturally linked to the
copula product
\begin{align}\label{def:uppprod}
D\vee E (u,v) := \int_0^1 \min\{\partial_2 D(u,t), \partial_2 E(v,t)\} \de t, \quad  (u,v)\in [0,1]^2.
\end{align}
Here, the copula derivatives are coupled with the upper Fr\'{e}chet copula \(M(u,v) = \min\{u,v\}\). Since \(M\) models comonotonicity, the integrand in \eqref{def:uppprod} models conditional comonotonicity of \(U\) and \(V\) given \(Z=t\). We therefore denote the copula product \(D\vee E\) as \emph{conditional comonotonicity product} or, simply, as \emph{upper product} of \(D\) and \(E\). Note that also the upper product itself is a copula.

While Markov products have been studied extensively in the copula literature,
in particular through their connections to Markov processes, Markov operators,
and idempotent copulas; see e.g.
\cite{darsow-1992,Olsen-1996,Durante-Klement-Quesada-Sarkoci-2007,
Mikusinski-Taylor-2009,Lageras-2010,Trutschnig-2013,
Fernandez-Trutschnig-Tschimpke-2021}
and \cite[Chapter~5]{fdsempi2016},
there is comparatively little literature on upper products.
To the best of our knowledge, upper products of copulas are studied mainly in \cite{Ansari-2019} and
\cite{Ansari-Rueschendorf-2018,Ansari-2021} in the context of risk bounds for partially specified factor models. In \cite{Shaked-2013}, their construction appears for verifying global dependence stochastic orderings.
The concept of conditional comonotonicity---which underlies upper products---has gained increasing attention in recent years, particularly in connection with time-dependent optimal transport problems.
For instance, the Knothe–Rosenblatt coupling---which realizes in many standard models the adapted Wasserstein distance between stochastic processes---is constructed via conditionally comonotone random variables; see \cite{Backhoff-Beiglboeck-Lin-Zalashko-2017,Backhoff-Kallblad-Robinson-2025,Rueschendorf-1985}.
Hence, upper products in \eqref{def:uppprod} are naturally linked to the dependence structure underlying the Knothe-Rosenblatt coupling.

\subsection{The upper product transform \(\cT\)}

In the present paper, we investigate a specific upper product of the form
\begin{align}\label{def:operatorT}
    \cT(C) :=  C\vee \Pi.
\end{align}
On the one hand, our aim is to motivate a more detailed study of conditional comonotonicity products since, as we show, they naturally occur in the context of dependence measures.
On the other hand, we aim to establish a bridge between two recently studied classes of dependence measures---Wasserstein correlations in \cite{wiesel2022} and rearranged dependence measures in \cite{strothmann2022}.
As we show in Theorem \ref{cor_TWR} below, both classes can be described by the \emph{upper product transform} \(\cT\) in \eqref{def:operatorT}.

Before introducing the concepts of Wasserstein correlations and rearranged dependence measures, we state our first main result due to which the upper product transform \(\cT\) characterizes independence and perfect functional dependence.
Further, the upper product transform \(\cT\) acts as a projection onto the class of SI copulas.
Recall that a bivariate random vector \((V,U)\) or its copula \(C\) is said to be \emph{stochastically increasing} (SI), if, for all \(v\in (0,1)\), \(P(V\geq v\mid U=u)\) is increasing in \(u\). Note that SI a positive-dependence concept that implies positive quadrant dependence, i.e.,
\begin{align}\label{eq_si_pqd}
    \Pi(u,v) \leq C(u,v) \leq M(u,v) \qquad \text{for all } C\in \cC^\uparrow \text{ and } (u,v)\in [0,1]^2,
\end{align}
see e.g. \cite[Section 3.10]{Muller-2002}. Here, \(\cC^\uparrow\) denotes the class of bivariate SI copulas. Moreover, for a copula \(C\in \cC\), we write \(C^\uparrow\) for its \emph{increasing rearranged copula} in \eqref{def:inc_rearranged_copula}, which is SI by construction.

\begin{theorem}\label{corcharT}
    The following statements hold true:
    \begin{enumerate}[label = (\roman*)]
    \item \label{corcharT1} \(\{ \cT(C)\mid C \in \cC\} = \{ C^\uparrow \mid C\in \cC\} = \cC^\uparrow\),
        \item \label{corcharT2} For \(C\in \cC\), it is \(C = \Pi\) if and only if \(\cT(C) = M\) if and only if \(\cT^2(C) = \Pi\),
        \item \label{corcharT3} \(C\in \cC\) is a perfect dependence copula if and only if \(\cT(C) = \Pi\) if and only if \(\cT^2(C) = M\).
    \end{enumerate}
\end{theorem}

By the above theorem, the upper product transform \(\cT\) characterizes the extreme dependence structures that underlie the values \(0\) and \(1\) of dependence measures satisfying Axioms \ref{prop1}--\ref{prop3}. For an illustration of the mapping \(\cT\), we refer to Figure \ref{Fig2}.

\begin{figure}
\begin{center}
    \begin{tikzpicture}
\tikzset{>={Stealth[length=8pt, width=5pt]}}

\tikzset{
    T_arrow/.style={
        ->, thick, dashed,
        decoration={
            markings,mark=at position 0.5 with {\arrow{>}}
        },
        postaction={decorate}
    }
}

\draw[ultra thick] (0,0) circle (3cm);
\draw[ultra thick, blue] (50:1.5cm) arc [start angle=50, end angle=130, radius=1.5cm];
\node[scale=2.5] at (1,-1) {$\cC$};

\draw[-, thick, red] (0,1.5) ellipse (1cm and 1.5cm);
\node[scale=1.5, red] at (0.5, 2.2) {$\cC^\uparrow$};
\node[scale=1.2, blue] at (-1.3, 1.1) {$\cC_\cT$};

\filldraw (90:3cm) circle (2pt);
\node[above] at (90:3cm) {\(M\)};

\filldraw (0,0) circle (2pt);
\node[below right] at (0,0) {$\Pi$};

\coordinate (C_point) at (-1,-2);
\filldraw (C_point) circle (2pt);
\node[below right] at (C_point) {$D$};

\coordinate (C_up_point) at (-0.2,2.3);
\filldraw (C_up_point) circle (2pt);
\node[above right] at (C_up_point) {$D^\uparrow$};

\coordinate (CvPi_point) at (-0.1,0.6);
\filldraw (CvPi_point) circle (2pt);
\node[below] at (CvPi_point) {~~$D\vee \Pi$};

\draw[T_arrow, bend left=70]
      (C_point) to node[midway, left=4pt] {\(\cT\)} (CvPi_point);

\draw[T_arrow, bend right=45]
      (CvPi_point) to node[pos=0.33, right=2pt] {\(\cT\)} (C_up_point);

\draw[T_arrow, bend right=50]
      (C_up_point) to node[pos=0.4, left=2pt] {\(\cT\)} (CvPi_point);
\end{tikzpicture}
\end{center}
\caption{Visualization of the operator \(\cT\colon \cC\to \cC^\uparrow\) in \eqref{def:operatorT}, which maps an arbitrary copula \(D\in \cC\) to the upper product \(D\vee \Pi\). Further, according to Theorem \ref{prop13}, \(\cT\) permutes the upper product \(D\vee \Pi\) and the rearranged copula \(D^\uparrow\). The outer circle represents the class \(\cC\) of bivariate copulas, while the red ellipse illustrates the subclass \(\cC^\uparrow\) of SI copulas where the mappings \(\cT\) and \(\cS\) coincide; see Corollary \ref{cor_TeqS}. The blue circular segment represents the set \(\cC_\cT\) in \eqref{def:C_T} of fixed points of \(\cT\).}
\label{Fig2}
\vspace{-3mm}
\end{figure}

\subsection{Wasserstein correlations and rearranged dependence measures}

We now consider two classes of dependence measures that are closely linked to Theorem \ref{corcharT}.
For the first class, let \(c\colon \R^2 \to \R\) be a function of the form
\begin{align}\label{def:convexcost}
    c(y,y') = h(y'-y),
\end{align}
where \(h\colon \R\to \R\) is convex, strictly convex at \(0\), and satisfies \(h(0) = 0\).
Then the \emph{Wasserstein correlation} with \emph{convex costs} \(c\) is defined as
\begin{align}\label{def:WassersteinCorrelation}
    \dW_c(Y,X)
    = \frac{\int_0^1 \cW_c(P^{Y|X=x}, \cU(0,1)) \de x}{\int_0^1\int_0^1 c(y,y') \de y \de y'},
\end{align}
where \(\cU(0,1) = P^Y\) denotes the uniform distribution on \((0,1)\), and
\(\cW_c(\nu,\nu')\) is the optimal transport cost between two distributions \(\nu\) and \(\nu'\) on \(\R\); see \eqref{def:OT}. Similar to the Wasserstein correlations with metric costs studied in \cite{wiesel2022}, \(\dW_c\) is a dependence measure that satisfies Axioms \ref{prop1}--\ref{prop3}; see \cite{Ansari-Fuchs-2026}.
Since Wasserstein correlations are, by construction, closely related to the concept of adapted Wasserstein distance, their connection to conditional comonotonicity and upper products is quite natural as follows.

\begin{proposition}[Representation of \(\dW_c\)]\label{prop:repdW}
Let \(Y,X,U,V\sim \cU(0,1)\) be random variables with \(X\) and \(V\) independent of \(U\).
\begin{enumerate}[label=(\roman*)]
    \item \label{prop:repdW1} The Wasserstein correlation in \eqref{def:WassersteinCorrelation} admits the representation
\begin{align}\label{eq:solWcor}
    \dW_c(Y,X) = \frac{\E c(F_{Y|X}^{-1}(U),U)}{\E c(U,V)}.
\end{align}
\item \label{prop:repdW2} The distribution function of \((F_{Y|X}^{-1}(U), U)\) is the upper product transform \(\cT(C_{Y,X})\).
\end{enumerate}
\end{proposition}

As a second class, we consider rearranged dependence measures. Their connection to the concepts of conditional comonotonicity and upper products is less apparent.
Rearranged dependence measures are defined with respect to a functional \(\mu\) that is assumed to satisfy the Axioms \ref{prop1}--\ref{prop3} on the subclass \(\cC^\uparrow\) of bivariate SI copulas.\footnote{That is, \(\mu(C)\) takes values in \([0,1]\) for all \(C\in \cC^\uparrow\) with \(\mu(C) = 0\) if and only if \(C= \Pi\), and \(\mu(C) = 1\) if and only if \(C=M\).}
Recall that \(C^\uparrow\in \cC^\uparrow\) is the (uniquely determined) increasing rearranged copula of \(C\in \cC\); see \eqref{def:inc_rearranged_copula}.
Now, the \emph{rearranged dependence measure} \(\sR_\mu\) is defined by
\begin{align}\label{eq:RearrDepMeasure}
    \sR_\mu(Y,X) := \mu(C_{Y,X}^\uparrow),
\end{align}
and it satisfies Axioms \ref{prop1}--\ref{prop3}; see \cite[Theorem 2.4]{strothmann2022}.
Examples of \(\mu\) are Spearman's \(\rho\), Kendall's \(\tau\), Gini's \(\gamma\), and Schweizer-Wolff measures \cite[Section 2.2]{strothmann2022} as well as standardized linear functionals of the form \(C\mapsto \int f \de C\) for supermodular functions \(f\); see \eqref{eq:R_lineara}.

The following main result relates upper product transforms \(\cT(C)\) and increasing rearranged copulas \(C^\uparrow\), which underlie Wasserstein correlations in \eqref{eq:solWcor} and rearranged dependence measures in \eqref{eq:RearrDepMeasure}, respectively. We refer to Figure \ref{fig3} for a visualization.

\begin{theorem}\label{prop13}
    We have \(\cT^2(C) = C^\uparrow\) and \(\cT(C^\uparrow) = \cT(C)\) for all \(C\in \cC\).
\end{theorem}

\begin{figure}[tbp]
    \centering
    \begin{subfigure}{0.7\textwidth}
        \centering
        \includegraphics[width=\textwidth, trim= 30pt 10pt 30pt 10pt, clip]{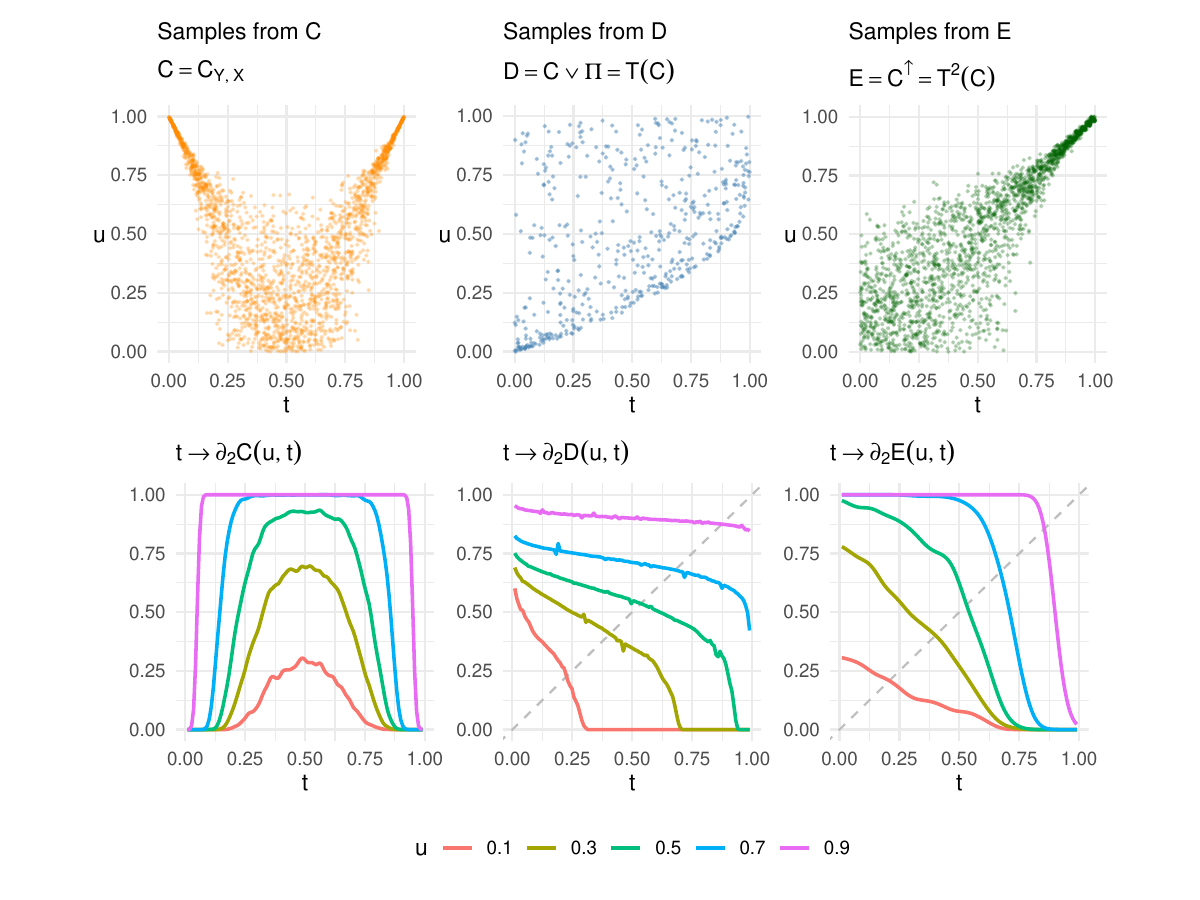}
    \caption{\(Y = X^2 + 0.5 \varepsilon\), where \(X\) is \(t\)-distributed with \(3\) degrees of freedom and independent of \(\varepsilon\sim N(0,1)\).}
    \end{subfigure}

    \begin{subfigure}{0.7\textwidth}
        \centering
        \includegraphics[width=\textwidth, trim= 30pt 10pt 30pt 0pt, clip]{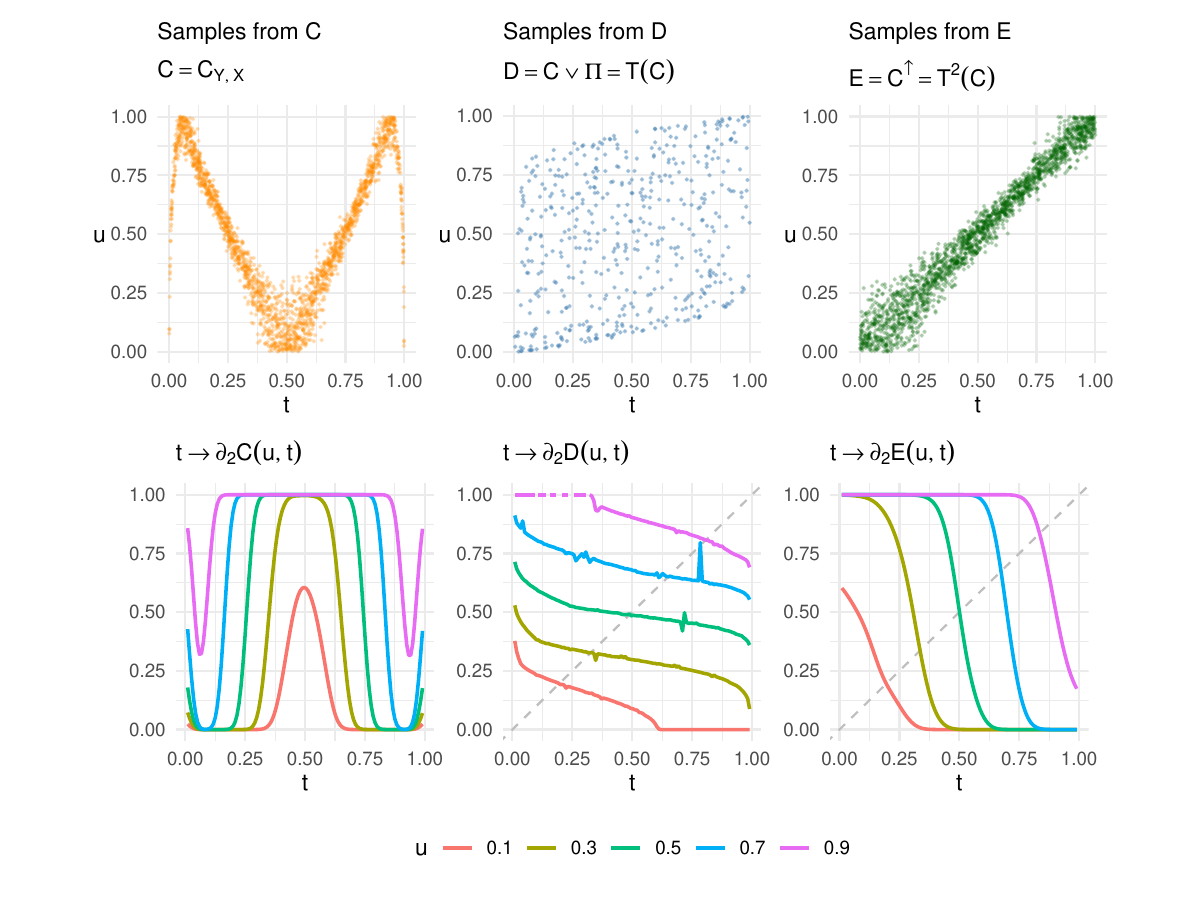}
    \caption{\(Y = \cos(2X) + 0.1 \varepsilon\), where \(X\) and \(\varepsilon\) are i.i.d. standard normal.}
    \end{subfigure}
    \caption{Samples from the copula \(C\) of a bivariate random vector \((Y,X)\) (left column), its upper product transform \(\cT(C) = C\vee \Pi\) (middle column), and its increasing rearrangement \(\cT^2(C) = C^\uparrow\) (right column). Due to Proposition \ref{thm:main_trafo} and Corollary \ref{cor_TeqS}, \(\cT\) permutes \(C\vee \Pi\) and \(C^\uparrow\) by reflecting their partial derivatives (with respect to the second component) at the main diagonal. }
    \label{fig3}
\end{figure}

As a consequence of Propsition \ref{prop:repdW} and Theorem \ref{prop13}, Wasserstein correlations and rearranged dependence measures admit a representation via the upper product transform \(\cT\) as follows.

\begin{theorem}\label{cor_TWR}
For \(C\in \cC\), let \((Y,X)\sim C\). Then we have
    \begin{align}\label{cor:T_rep}
        \dW_c(Y,X)
        &= \frac{\int_{[0,1]^2} c(u,u') \,\de \cT(C)(u,u')}{\int_0^1 \int_0^1 c(u,u') \de u' \de u} \quad
        \text{and} \qquad \sR_\mu(Y,X) = \mu(\cT^2(C)).
        \end{align}
In particular, \(\dW_c(Y,X) = 0\) (resp. \(1\)) if and only if \(\cT(C) = M\) (resp. \(\Pi\)), and \(\sR_\mu(Y,X) = 0\) (resp. \(1\)) if and only if \(\cT^2(C) = \Pi\) (resp. \(M\)).
\end{theorem}

\begin{remark}\label{rem_18}
\begin{enumerate}[label = (\alph*)]
    \item There are further dependence measures that satisfy Axioms \ref{prop1}--\ref{prop3} and admit a representation in terms of \(\cT\). For instance, the dependence measure \(\zeta_1\) in \cite{Trutschnig-2011} can be expressed via Spearman's footrule as
\begin{align*}
    \nonumber \zeta_1(C) &:= 3 \int_0^1\int_0^1 \left| \partial_2 C(u,t) - u\right| \de t \de u \\
    \nonumber &\phantom{:}= 3 \int_0^1 \int_0^1 \partial_2 C(u,t) + u - 2 \min\{\partial_2 C(u,t) , u \} \de t \de u \\
    &\phantom{:}= 3 - 6 \int_0^1 C\vee \Pi(u,u) \de u =  1 - \psi(\cT(C)).
\end{align*}
In particular, we obtain from the integral identity in \eqref{eq_form_gen_inv} that \(\zeta_1\) coincides with the \(1\)-Wasserstein correlation, i.e. for the cost function \(c(u,v) = |v-u|\) and for
\((Y,X)\sim C\), we have
\begin{align*}    \zeta_1(C) = 3\int_0^1 \int_0^1 | F_{Y|X=t}(u) - u | \de u \de t = 3 \int_0^1 \int_0^1 | F_{Y|X=t}^{-1}(u) - u | \de u \de t = \dW_1(Y,X).
\end{align*}
\item Another class of dependence measures that includes \(\zeta_1\) and Chatterjee's \(\xi\) is based on functionals of the form
\begin{align}\label{eq_rep_chat_phi}
\begin{split}
    &\int_0^1 \int_0^1 \varphi(F_{Y|X=t}(v) - v) \de t \de v = \int_0^1 \int_0^1 \varphi(\partial_2 C(v,t) - v) \de t \de v \\
    &= \int_0^1 \int_0^1 \varphi(\partial_2 C^\uparrow(v,t) - v) \de t \de v = \int_0^1 \int_0^1 \varphi(\partial_2 \cT^2(C)(v,t) - v) \de t \de v,
\end{split}
\end{align}
for convex functions \(\varphi\),
where the second equality follows with \eqref{def:intrearr}.
Similarly, the measures of sensitivity in \cite{Ansari-LFT-2023} satisfy Axioms \ref{prop1}--\ref{prop3} and are of the form
\begin{align}\label{eq_rep_chat_phi2}
\begin{split}
    &\int_0^1 \int_0^1 \int_0^1 \varphi(F_{Y|X=t}(v) - F_{Y|X=s}(v)) \de s \de t \de v \\
    &= \int_0^1 \int_0^1 \int_0^1 \varphi(\partial_2 \cT^2(C)(v,t) - \partial_2 \cT^2(C)(v,s)) \de s \de t \de v.
\end{split}
\end{align}
Hence, also the rank-based dependence measures in \eqref{eq_rep_chat_phi} and \eqref{eq_rep_chat_phi2} can be represented via the upper product transform \(\cT\).
This demonstrates that upper products with $\Pi$ underlie various classes of dependence measures, whereas the connection between Chatterjee's $\xi$ and the Markov product in \eqref{rep_xi} is rather exceptional.
\item Recall that Wasserstein correlations and rearranged dependence measures in Theorem \ref{cor_TWR}, as well as rank-based dependence measures in \eqref{eq_rep_chat_phi} and \eqref{eq_rep_chat_phi2} like Chatterjee's \(\xi\), all admit representations in terms of \(\cT\). Hence, since \(\cT\) maps surjectively onto \(\cC^\uparrow\) (see Theorem \ref{corcharT}\,\ref{corcharT1}), the dependence structure relevant to quantify the strength of functional dependence of \(Y\) on \(X\) can be reduced to SI copulas.
In other words, for such dependence measures, all relevant dependence information is captured within the subclass of SI copulas.
\end{enumerate}
\end{remark}

\subsection{Dependence measures from an order-theoretic perspective}

To understand the close connection between dependence measures and the concept of conditional comonotonicity, and to see why the conditional
independence structure underlying Chatterjee's \(\xi\) is rather special, let
us view dependence measures from the perspective of stochastic orders.

Typically, it is not just a single functional that satisfies the axioms of a
dependence measure, but rather an entire class of related functionals. The
underlying structure of dependence can often be captured by a suitable
stochastic order, which in turn gives rise to a corresponding class of
dependence measures. To make this principle precise, we formulate the following
meta-result. Any dependence order \(\preccurlyeq\) considered below is assumed
to be reflexive and transitive, but not necessarily antisymmetric.\footnote{A dependence order \(\preccurlyeq\) is \emph{reflexive} if \(C\preccurlyeq C\) holds true for all \(C\). It is \emph{transitive} if \(C\preccurlyeq D\) and \(D\preccurlyeq E\) implies \(C\preccurlyeq E\). It is \emph{antisymmetric} if \(C\preccurlyeq D\) and \(D\preccurlyeq C\) implies \(C=D\).}

\begin{lemma}\label{prop_meta} Assume that \(\preccurlyeq\) is a dependence order on \(\cC\) that characterizes \begin{enumerate}[label = (\roman*)] \item \label{prop_meta1} independence, i.e., \(C\preccurlyeq D\) for all \(D\in \cC\) if and only if \(C = \Pi\), and \item \label{prop_meta2} perfect (functional) dependence, i.e., \(D\preccurlyeq E\) for all \(D\in \cC\) if and only if \(E\) is a perfect dependence copula. \end{enumerate} Let \(\mu\colon \cC\to [0,1]\) be surjective and \(\preccurlyeq\)-increasing (i.e., \(C\preccurlyeq D\) implies \(\mu(C) \leq \mu(D)\)) with strict inequality at independence and perfect dependence (i.e., \(\Pi\preccurlyeq D\) and \(D\not\preccurlyeq \Pi\) implies \(\mu(\Pi) < \mu(D)\), and \(C\preccurlyeq D\) and \(D\not \preccurlyeq C\) for any perfect dependence copula \(D\) implies \(\mu(C) < \mu(D)\)). Then the functional \(\mu(Y,X): = \mu(C)\), \((Y,X)\sim C\), satisfies Axioms \ref{prop1}--\ref{prop3}. \end{lemma}

Motivated by the representation of \(\xi\) in Eq. \eqref{rep_xi}, a dependence order underlying Chatterjee's rank correlation may be defined via the diagonal of Markov products. To be precise, the diagonal of \(C\ast C\) satisfies the inequality
\begin{align}\label{eq_diag}
    u^2 =\Pi(u,u) \leq C\ast C\,(u,u) \leq M(u,u) = u \quad \text{for all } u\in [0,1].
\end{align}
The bounds in \eqref{eq_diag} are attained for all \(u\in [0,1]\) exactly in the case of independence and perfect dependence, respectively, i.e.
\begin{align*}
    u^2 = C\ast C(u,u) \quad \text{for all } u\in [0,1]\quad &\Longleftrightarrow \quad C = \Pi,\\
     u= C\ast C(u,u) \quad \text{for all } u\in [0,1]\quad &\Longleftrightarrow \quad C \text{ is a perfect dependence copula};
\end{align*}
see \cite{Fuchs-2024}. This gives the following result.

\begin{proposition}\label{cor_lc12}
    For \(C,C'\in \cC\), define \(C\preccurlyeq_1 C'\) by \(C\ast C(u,u) \leq C'\ast C'(u,u)\) for all \(u\in [0,1]\). Then \(\preccurlyeq_1\) satisfies conditions \ref{prop_meta1} and \ref{prop_meta2} of Lemma \ref{prop_meta}.
\end{proposition}
Using that Spearman's footrule is a normalized functional that is strictly increasing
with respect to the pointwise order on copula diagonals, the map
\(C \mapsto \psi(C\ast C)=\xi(C)
\)
is surjective onto \([0,1]\) and strictly \(\preccurlyeq_1\)-increasing. Thus,
Proposition~\ref{cor_lc12} recovers the well-known fact that \(\xi\) satisfies Axioms~\ref{prop1}--\ref{prop3}. More generally, suitably
weighted functionals of the diagonal of \(C\ast C\) also characterize
independence and perfect dependence.

However, the class of dependence measures that can be generates via the \(\preccurlyeq_1\)-order in this way
is comparatively limited. Indeed, away from the diagonal, the Markov product
no longer characterizes independence through a pointwise lower bound: there
exist \(C\in\cC\) and \((u,v)\in[0,1]^2\) such that
\begin{align}\label{eq_cop_diagn}
    C\ast C(u,v)< \Pi(u,v)=uv;
\end{align}
see \cite{Fuchs-2024}. In contrast, the upper-product transform \(\cT\) allows one to consider
functionals defined on the entire unit square \([0,1]^2\), thereby giving rise
to substantially richer classes of dependence measures. The following result
is a direct consequence of Theorem~\ref{corcharT} and
Eq.~\eqref{eq_si_pqd}.

\begin{proposition}\label{cor_C2_C3}
     For \(C,C'\in \cC\), define the dependence order
     \begin{enumerate}[label = (\roman*)]
         \item \(C\preccurlyeq_2 C'\) by \(\cT(C)(u,v) \geq \cT(C')(u,v)\) for all \((u,v)\in [0,1]^2\),
         \item \(C\preccurlyeq_3 C'\) by \(\cT^2(C)(u,v) \leq \cT^2(C')(u,v)\) for all \((u,v)\in [0,1]^2\).
     \end{enumerate}
      Then \(\preccurlyeq_2\) and \(\preccurlyeq_3\) satisfy Conditions \ref{prop_meta1} and \ref{prop_meta2} of Lemma \ref{prop_meta}.
\end{proposition}
Using that \(\cT(C)\) is SI for all \(C\), any normalized copula functional that is strictly increasing/decreasing on \(\cC^\uparrow\) with respect to the pointwise order of copulas implies a dependence measure that satisfies \ref{prop1}--\ref{prop3}. Important examples are Wasserstein correlations and rearranged dependence measures discussed above. Further examples include integrals of supermodular and submodular functions, respectively; see Eq. \eqref{eq_lo_supsub}.
In Theorem \ref{prop:ord_T}, we show that the recently investigated, rearrangement-based \(\leq_{\partial_2 S}\)-order is equivalent to the orders \(\preccurlyeq_2\) and \(\preccurlyeq_3\) and implies the order \(\preccurlyeq_1\). In particular, the \(\leq_{\partial_2 S}\)-order underlies Chatterjee's rank correlation, Wasserstein correlations, rearranged dependence measures, and the measures of sensitivity  of the type \eqref{eq_rep_chat_phi} and \eqref{eq_rep_chat_phi2}. We refer to
Section \ref{sec5} for details.

\subsection{Structure of the paper}

The rest of the paper is organized as follows.
To establish the link between Wasserstein correlations and rearranged dependence measures via the upper product transform \(\cT\), we introduce in Section \ref{sec2} a new copula operator \(\cS\). This operator serves as an auxiliary tool in our analysis of the structural properties of \(\cT\).
It transforms an SI copula into another SI copula by reflecting the partial derivatives at the main diagonal; see Definition \ref{def_reflectS} and Figure \ref{fig:S_Gauss}. As we show in Proposition \ref{thm:main_trafo}, \(\cS\) permutes upper products \(C\vee \Pi\) and rearranged copulas \(C^\uparrow\), and thus it connects Wasserstein correlations with rearranged dependence measures; see Corollaries \ref{thm:RepWasserstein} and \ref{cor:sr}.

In Sections \ref{sec3}-\ref{sec6}, we study the upper product transform \(\cT\) in more detail.
Due to Corollary \ref{cor_TeqS}, \(\cT\) reduces to the reflection operator \(\cS\) when it is restricted to SI copulas, and thus \(\cT\) generalizes \(\cS\).
In Theorem \ref{the_metric}, we establish contraction and isometry properties of \(\cT\).
Noting that \(\cT(M) = \Pi\) and \(\cT(\Pi) = M\) (recall Theorem \ref{corcharT}), we study in Section \ref{sec4} necessary and sufficient conditions for fixed points of \(\cT\), i.e. for copulas that satisfy \(\cT(C) = C\). Such fixed points are interesting because they yield the identity \(\dW_c + \sR_{-c} = 1\); see Proposition \ref{prop:motiv_fp}.
In Theorem \ref{lem:SI_cop_construction}, we construct infinitely many SI copulas that are fixed points of \(\cT\). The underlying construction method is based on convolutions and may be of independent interest.
In Section \ref{sec5}, we describe the behavior of the transformations \(\cT\) and \(\cT^2\) through the recently studied \(\leq_{\partial_2 S}\)-order, a dependence order that can be verified in many copula models; see \cite{Ansari-Rockel-2023}.
In Section \ref{sec6}, we provide copula-based continuity conditions for \(\cT\) which, in turn, yield continuity properties for Wasserstein correlations and rearranged dependence measures. This is particularly relevant since dependence measures satisfying Axioms \ref{prop1} -- \ref{prop3} fail to be weakly continuous; see \cite{buecher2024}.
The proofs of Section \ref{secIntro} rely on the concepts developed in Sections \ref{sec3}--\ref{sec6} and are provided in Section \ref{sec_pr1}.
We conclude in Section \ref{sec7} by highlighting the role of the upper product transform $\cT$ as a fundamental concept that underlies the construction of dependence measures.

\section{Rearrangements and upper products}\label{sec2}

In Section, we establish a connection between increasing rearranged copulas and the upper product transform \(\cT\) via a new reflection operator \(\cS\). First, however, we recall the basic dependence concepts relevant to this paper.

\subsection{Basic dependence concepts}

A bivariate copula \(C\) is a distribution function on \([0,1]^2\) with standard uniform marginals.
By Sklar's Theorem, every bivariate distribution function \(F\) on \(\R^2\) can be decomposed into
\begin{align}\label{eq_sklar}
    F(x_1,x_2) = C (F_1(x_1),F_2(x_2)), \qquad (x_1,x_2)\in \R^2,
\end{align}
where \(C\) is a bivariate copula and where \(F_1\) and \(F_2\) denote the first and second marginal distribution functions of \(F\); see e.g. \cite{Nelsen-2006}.
If \(F\) is continuous, the copula \(C\) in \eqref{eq_sklar} is unique and thus fully captures the dependence structure of \(F\).
Recall that \(\cC\) denotes the class of bivariate copulas and \(\cC^\uparrow\) the class of bivariate SI copulas. The independence copula \(\Pi(u,v) = uv\) and the upper Fr\'{e}chet copula \(M(u,v) = \min\{u,v\}\) are SI, and they describe the extreme cases of positive dependence. Specifically, \(M \geq_{lo} C\) for all \(C\in \cC\), while \(\Pi \leq_{lo} C\) holds for all \(C\in \cC^\uparrow\). Here, \(\leq_{lo}\) denotes the \emph{lower orthant order} (i.e., the pointwise order) of copulas, that is, for \(D,E\in \cC\), \(D\leq_{lo} E\) is defined by \(D(u,v)\leq E(u,v)\) for all \((u,v)\in [0,1]^2\).

Throughout this work, we make frequent use of copula derivatives, adopting the following characterization of copulas; see e.g. \cite[Lemma 1]{Ansari-Rockel-2026}.

\begin{lemma}[A characterization of copulas]\label{charSIcop}
    A function \(C\colon [0,1]^2 \to [0,1]\) is a bivariate copula if and only if there exists a family \((h_v)_{v\in [0,1]}\) of measurable functions \(h_v\colon [0,1]\to [0,1]\) such that
    \begin{enumerate}[label=(\roman*), font=\upshape]
        \item \label{charSIcop1} \(C(v,u) = \int_0^u h_v(t) \de t\) for all \((u,v)\in [0,1]^2,\)
        \item \label{charSIcop2} \(h_v(t)\) is increasing in \(v\) for all \(t\in [0,1],\)
        \item \label{charSIcop3} \(\int_0^1 h_v(t) \de t = v\) for all \(v\in [0,1]\).
    \end{enumerate}
\end{lemma}

\begin{remark}
\begin{enumerate}[label=(\alph*), font=\upshape]
    \item
    For \(X,Y,Z\sim \cU(0,1)\), the copula of \((Y,X)\) can be written as a product of the copulas \(D:= C_{Y,Z}\), \(E:= C_{X,Z}\), and \(C_t := C_{(Y,X)\mid Z=t}\), \(t\in (0,1)\), through
    \begin{align}\label{eq:gencopprod}
        C_{Y,X}(u,v) = \int_0^1 C_t\left(\partial_2 D(u,t), \partial_2 E(v,t)\right) \de t,
    \end{align}
see \cite[Theorem 2.7]{Ansari-2021}. If, for all \(t\in [0,1]\), \(C_t = \Pi\) or, equivalently, \(Y\) and \(X\) are conditionally independent given \(Z=t\), the right-hand side of \eqref{eq:gencopprod} reduces to the Markov product \(D\ast E\) in \eqref{def:stern_product}.
If, for all \(t\in [0,1]\), \(C_t = M\) or, equivalently, \(Y\) and \(X\) are conditionally comonotone given \(Z=t\), the right-hand side of \eqref{eq:gencopprod} is the upper product \(D\vee E\) in \eqref{def:uppprod}.
Since \(C_t\leq_{lo} M\) for all \(t\), it immediately follows that \(C_{Y,X}\leq_{lo} D\vee E\) (in particular, \(D\ast E \leq_{lo} D\vee E\)). Hence, upper products describe the worst-case dependence structure in models where the bivariate marginal distributions \((X,Z)\) and \((Y,Z)\) are known but the conditional dependencies \((X,Y)\mid Z=t\) are not specified; see \cite{Ansari-Rueschendorf-2018,Bernard-Rueschendorf-Vanduffel-Wang-2017}.
\item The function \(h_v(t)\) in Lemma \ref{charSIcop} coincides with the copula derivative \(\partial_2 C(v,t)\) for Lebesgue-almost all \((v,t)\in [0,1]^2\). Hence, \((h_v)_{v\in [0,1]}\) determines a family of conditional distribution functions. Note that, in general, neither \(v\mapsto h_v(t)\) nor \(v\mapsto \partial_2 C(v,t)\) is right-continuous.
To avoid redundant notation, we utilize copula derivatives rather than right-continuous modifications or Markov kernels.
This approach is sufficient for our purposes, as any discrepancies occur only on null sets and thus do not affect the relevant integrals.
\end{enumerate}
\end{remark}

The construction of the rearranged dependence measure \(\sR_\mu\) in \eqref{eq:RearrDepMeasure} is based on the concept of rearrangements. To be precise, the \emph{(increasing) rearranged copula} \(C^\uparrow\) of a bivariate copula \(C\) is defined by
\begin{align}\label{def:inc_rearranged_copula}
C^\uparrow(v,u) := \int_0^u h_v^*(t) \de t \quad \text{for } h_v(t) = \partial_2 C(v,t);
\end{align}
see \cite{Ansari-2021,strothmann2022}.
Here, \(h^*\) denotes the \emph{decreasing} rearrangement of a measurable function \(h\colon (0,1)\to [0,1]\), i.e. the (essentially with respect to the Lebesgue measure \(\lambda\)) uniquely determined decreasing function \(h^*\) such that
\begin{align}\label{def:intrearr}
    \lambda( h^*\leq y) = \lambda(h \leq y) \quad \text{for all } y \in [0,1].
\end{align}
Note that \(h_v^*\) is increasing in \(v\) but decreasing in \(t\).
Since for all \(u\in [0,1]\), \(t\mapsto \partial_2 C^\uparrow(u,t)\) is decreasing outside a \(\lambda\)-null set, \(C^\uparrow\) is SI; see \cite[Proposition 3.17]{Ansari-2021}. Hence, \(\sR_\mu\) in \eqref{eq:RearrDepMeasure} is well-defined. An overview of well-known families of SI copulas is given in \cite{Ansari-Rockel-2023}. Note that \(C\leq_{lo} C^\uparrow\) with equality if and only if \(C\) is SI, that is,
\begin{align}\label{eq:CCSI}
    C = C^\uparrow \qquad \Longleftrightarrow \qquad C\in \cC^\uparrow;
\end{align}
see \cite[Proposition 3.17]{Ansari-2021}.

We make use of the following lemma which gives some elementary properties of upper products.

\begin{lemma}[Some properties of upper products]\label{lem:propuppprod}
For all \(C,D\in \cC\), we have
\begin{enumerate}[label=(\roman*)]
\item \label{lem:propuppprod1} \(C\vee M = C\),
\item \label{lem:propuppprod2} \(C \vee D = M\) if and only if \(C=D\),
\item \label{lem:propuppprod3} \(C\vee \Pi\) is SI.
\end{enumerate}
\end{lemma}

\begin{proof}    Statements \ref{lem:propuppprod1} and \ref{lem:propuppprod2} are given in \cite[Proposition 2.4]{Ansari-Rueschendorf-2018}. For the proof of statement \ref{lem:propuppprod3},
    fix \(v\in (0,1)\) and set \(f_v = \partial_2 C(v,\cdot)\). Then, we obtain
    \begin{align}\label{eqlem:propuppprod}
    \begin{split}
        \partial_2 (C\vee \Pi)(v,u) &= \lim_{h\to 0} \int_0^1 \frac{\min \{ f_v(t), u+h\} - \min\{f_v(t) , u\}}{h} \de t \\
        &=  \int_0^1 \frac{d}{d u} \min \{f_v(t), u\} \de t
        = \int_0^1 \1_{\{ f_v(t) > u\}} \de t
        = \lambda(\{t \mid \partial_2 C(v,t) > u\}),
    \end{split}
    \end{align}
    where we apply dominated convergence for the second equality using that the integrand is bounded by \(1\). The third equality holds true for \(\lambda\)-almost all \(v\in (0,1)\) using that \(f_v\) can have at most countably many constant segments.
    Since the right-hand term of \eqref{eqlem:propuppprod} is decreasing in \(u\), the statement follows.
\end{proof}

\begin{remark}\label{rem_prop_upp_prod}
    Upper products behave fundamentally differently from Markov products. For instance, the latter satisfies
    \(C\ast M =  C\) and \( C \ast \Pi = \Pi\) while \(C\vee M = C\) and \(C\vee C = M\) by Lemma \ref{lem:propuppprod}.
    We refer to \cite{Ansari-Rueschendorf-2018} for further properties of upper products, while a comprehensive overview of Markov products can be found in \cite[Section 5]{fdsempi2016}.
\end{remark}

As noted in the introduction, Wasserstein correlations are closely linked to the concept of conditional comonotonicity and thus to upper products.
To make this precise, recall the classical result that, for submodular\footnote{A function \(f\colon \R^2 \to \R\) is \emph{submodular} (supermodular) if \(f(x,y) + f(x',y') \leq (\geq) f(x,y')+ f(x',y)\) for all \(x \leq x'\) and \(y \leq y'\).} cost functions \(c\), the comonotone coupling \((F_\nu^{-1}(U),F_{\nu'}^{-1}(U))\) solves the optimal transport problem between two distributions \(\nu\) and \(\nu'\) on \(\R\), i.e.
\begin{align}\label{def:OT}
    \cW_c(\nu,\nu') := \inf_{\pi \in \mathrm{Cpl}(\nu,\nu')} \int c(y,y') \de \pi(y,y') = \E c(F_\nu^{-1}(U),F_{\nu'}^{-1}(U));
\end{align}
see \cite[Theorem 3.1.2]{Ru-1998}.
Here, \(\mathrm{Cpl}(\nu,\nu')\) denotes the set of couplings between \(\nu\) and \(\nu'\), i.e., the set of distributions on \(\R^2\) with first and second marginal \(\nu\) and \(\nu'\), respectively. Further, \((0,1)\ni t\mapsto F_\nu^{-1} := \inf\{x\mid F_\nu(x) \geq t\}\) is the (left-continuous) generalized inverse distribution function associated with \(\nu\), and \(U\) is a random variable that is uniform on \((0,1)\).
Now, assume that \(U\) is independent of \(X\). Then, since the convex cost function \(c\) in \eqref{def:convexcost} is submodular,
the conditionally (on \(X=x\)) comonotone coupling \((F_{Y|X=x}^{-1}(U),F_Y^{-1}(U))\) realizes the cost \(\cW_c(P^{Y|X=x},P^Y)\) in the numerator of the Wasserstein correlation in \eqref{def:WassersteinCorrelation}. This yields the following representation in \eqref{eq:solWcor}.

\subsection{A new reflection operator for SI copulas}

As shown in Proposition \ref{prop:repdW}, optimal couplings for \(\dW_c\) are described by upper products with \(\Pi\). In the following, we show that the rearranged copulas in \eqref{def:inc_rearranged_copula}, and hence rearranged dependence measures, can likewise be represented via upper products of the form \(C\vee \Pi\).
To this end, we introduce a new reflection operator \(\cS\) on the class \(\cC^\uparrow\) of SI copulas as an auxiliary tool.
Later, in Section \ref{sec3}, we prove that the upper product transform \(\cT\) in \eqref{def:operatorT}, which is defined on the whole class of bivariate copulas, generalizes \(\cS\) since, as we show, it coincides with \(\cS\) on \(\cC^\uparrow\).
While the behavior of $\cT$ is not immediately obvious---neither in general nor specifically within the class $\cC^\uparrow$---the operator $\cS$ admits a straightforward interpretation as a reflection.
We denote by
\(f^{-1}\) the generalized inverse of a right-continuous, \emph{decreasing} function \(f\colon [0,1]\to [0,1]\) defined by
\begin{align}\label{defgeninv}
    f^{-1}(w) := \inf \{ t\in [0,1] \mid f(t) \leq w\}
\end{align}
with the convention that \(\inf\emptyset := 1\).
Recall that any SI copula \(C\) is concave in its second component and thus, for all \(v\in [0,1]\), \(t\mapsto \partial_2 C(v,t)\) is decreasing outside a \(\lambda\)-null set. In this case, we consider for convenience the right-continuous version \(t\mapsto \partial_2^+ C(v,t) := \lim_{s\downarrow t} \partial_2 C(v,s)\) of the partial derivative, which corresponds to the lower Dini derivative.

\begin{definition}[Reflection operator \(\cS\)]\label{def_reflectS}
    For \(C\in \cC^\uparrow\), define
\begin{align}\label{def:reflectop}
    \cS(C)(v,u) := \int_0^u (\partial_2^+ C(v,\cdot))^{-1}(t) \de t \quad \text{for } (v,u) \in [0,1]^2.
\end{align}
\end{definition}

\begin{remark}
\begin{enumerate}[label=(\alph*)]\label{remreflopte}
    \item \label{remrefloptea}
    It is important to note that we define the copula operator \(\cS\) by inverting the copula derivative \(\partial_2^+C(v,t)\) with respect to the second argument \(t\). For \((Y,X)\sim C\), this corresponds to an inverse of the conditional distribution function \(P(Y\leq y\mid X=t) = F_{Y|X=t}(v) = \partial_2^+ C(v,t)\) with respect to the \emph{conditioning} variable.
    As we illustrate in Figure \ref{fig:S_Gauss}, the operator \(\cS\) reflects the copula derivative \(t\mapsto \partial_2^+ C(v,t)\) along the main diagonal. It is not difficult to see that the reflected derivatives \(h_v := (\partial_2^+ C(v,\cdot))^{-1}\) define a family \((h_v)_{v\in [0,1]}\) of decreasing functions that satisfy conditions \ref{charSIcop1}--\ref{charSIcop3} of Lemma \ref{charSIcop}. Hence, \(\cS(C)\) is an SI copula.
    Formally, this is also a consequence of Proposition \ref{thm:main_trafo} and Lemma \ref{lem:propuppprod}\,\ref{lem:propuppprod3}.
    \item Classical reflection operators mirror the mass of copulas at the horizontal and vertical line \(u=1/2\) and \(v=1/2\), respectively, or at the diagonal \(u=v\). Such reflections lead to the concepts of survival copula and transposed copula; see \cite{fdsempi2016,Nelsen-2006}. As we show in this paper, our new reflection operator \(\cS\) in \eqref{def:reflectop} is well-motivated via dependence measures. In contrast to the classical reflections, its behavior is not intuitive; see Figure \ref{fig:S_Gauss} or the middle and left column of Figure \ref{fig3}. Note that, by Corollary \ref{cor_TeqS}, \(\cT(C) = \cS(C)\) for all SI copulas \(C\).
    \item \label{remreflopteb} In contrast to the construction in \eqref{def:reflectop}, it is well-known that taking the inverse with respect to the first component of \(\partial_2^+C(v,t)\) does not yield a copula.
    For example, consider the EFGM copula \(C(v,u) = vu + vu(1-v)(1-u)\); see \cite[Section 6.3]{fdsempi2016}. Then, for \((Y,X)\sim C\), the function \(v\mapsto F_t(v) = \partial_2 C(v,t) = v + v (1-v) (1-2t)\) is the conditional distribution function of \(Y\) given \(X=t\). The associated quantile function is \(F_{Y|X=t}^{-1}(v) = F_t^{-1}(v) = \left(1-t - \sqrt{(1-t)^2 -v(1-2t)}\right)/(1-2t)\) for \(t\ne 1/2\). For \(v = 0.25\), one obtains \(\int_0^1 F_t^{-1}(v) \approx 0.28\). Hence, \((u,v) \mapsto \int_0^u (\partial_2^+ C(\cdot,t))^{-1}(v) \de t\) does not define a copula as a consequence of Lemma \ref{charSIcop}.
    We conclude that, in general, there is no copula \(C^*\) such that, for \((Y^*,X^*)\sim C^*\), the identity \(F_{Y^*|X^*=x}^{-1}(v) = F_{Y|X=x}(v)\) is satisfied for Lebesgue-almost all \((v,x)\in [0,1]^2\).
\end{enumerate}
\end{remark}

The following proposition is of central importance for the main results presented in Section \ref{sec3}. It establishes a correspondence between upper products and rearranged copulas through the \emph{reflection} operator \(\cS\).
In particular, it shows that \(\cS(C)\) is an SI copula.

\begin{proposition}[Rearranged copulas and upper products]\label{thm:main_trafo}
    For \(C\in \cC\), the reflection operator \(\cS\) in \eqref{def:reflectop} satisfies
    \begin{enumerate}[label=(\roman*)]
        \item \label{thm:main_trafo1} \(\cS(C\vee \Pi) = C^\uparrow\),
        \item \label{thm:main_trafo2} \(\cS(C^\uparrow) = C\vee \Pi\).
    \end{enumerate}
\end{proposition}

\begin{proof}    \ref{thm:main_trafo1}: Fix \(v\in [0,1]\) and set \(h_v := \partial_2^+ C^\uparrow(v,\cdot)\). It follows that
    \begin{align*}
        \partial_2 (C\vee \Pi)(v,u)         = \lambda(\{t \mid \partial_2 C(v,t) > u\})
        = \lambda(\{t \mid  h_v(t) > u\}) = \lambda(\{t \mid  t < h_v^{-1}(u)\}) = h_v^{-1}(u),
    \end{align*}
    where the first equality is due to \eqref{eqlem:propuppprod}. For the second equality, we use that \(C^\uparrow\) is the increasing rearranged copula of \(C\) and thus \(h_v\) a rearrangement of \(\partial_2 C(v,\cdot)\).
    For the third equality, we use the fact that, for a decreasing and right-continuous function \(h\colon [0,1]\to [0,1]\), we have \(h(t) > u\) if and only if \(t < h^{-1}(u)\).     It follows that
    \begin{align*}
        h_v(t) = \left(\partial_2 (C\vee \Pi)(v,\cdot)\right)^{-1}(t)
    \end{align*}
    for \(\lambda\)-almost all \(t\in [0,1]\)
    and thus
    \begin{align*}
        \cS(C\vee \Pi)(v,u) = \int_0^u h_v(t) \de t = C^\uparrow(v,u).
    \end{align*}
    To prove \ref{thm:main_trafo2}, define \(g_v(t):= \partial_2^+ (C\vee \Pi)(v,t)\) for \(t\in (0,1)\). Similarly as above, we obtain \(g_v = (\partial_2^+ C^\uparrow (v,\cdot))^{-1}\). This implies
    \(\cS(C^\uparrow)(v,u) = \int_0^u g_v(t) \de t = C\vee \Pi(v,u)\).
\end{proof}

\begin{figure}[tbp]
    \centering
    \begin{subfigure}[b]{0.45\textwidth}
        \centering
        \includegraphics[width=\textwidth, trim=80pt 20pt 70pt 10pt, clip]{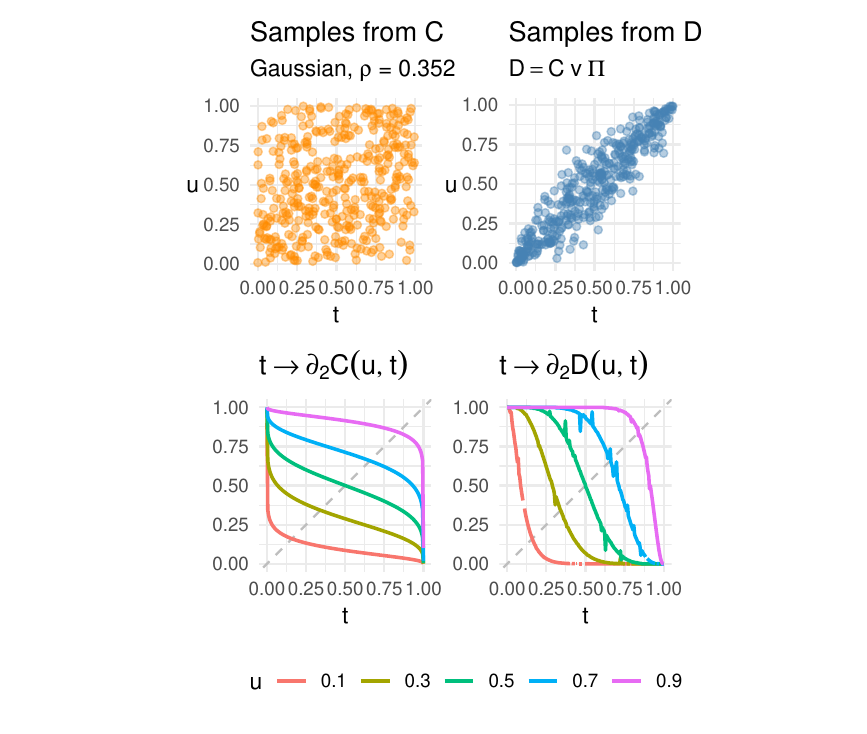}
        \caption{Sample (top) and partial derivative (bottom) of \(C_{0.352}^{\text{Gauss}}\) (left) and \(\cS(C_{0.352}^{\text{Gauss}}) = C_{0.936}^{\text{Gauss}}\) (right)}
        \label{fig:plot1}
    \end{subfigure}
    \hfill     \begin{subfigure}[b]{0.45\textwidth}
        \centering
        \includegraphics[width=\textwidth, trim=80pt 20pt 70pt 10pt, clip]{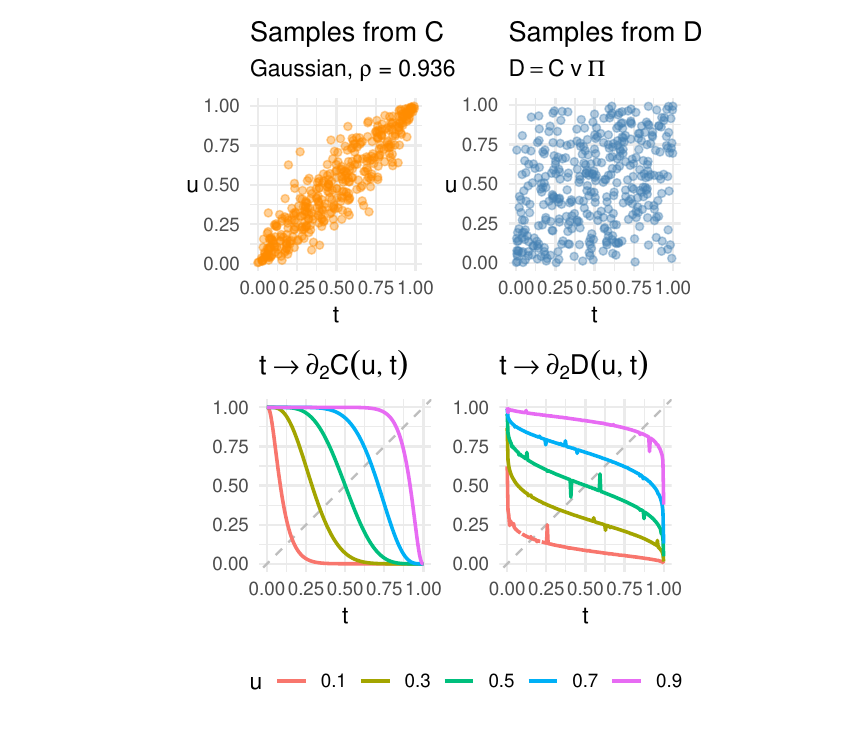}
        \caption{Sample (top) and partial derivative (bottom) of \(C_{0.936}^{\text{Gauss}}\) (left) and \(\cS(C_{0.936}^{\text{Gauss}}) = C_{0.352}^{\text{Gauss}}\) (right)}
        \label{fig:plot2}
    \end{subfigure}
    \caption{Visualization of the reflection operator \(\cS\colon \cC^\uparrow\to \cC^\uparrow\) in \eqref{def:reflectop} for the Gaussian copula \(C_{\varrho}^{\text{Gauss}}\) with parameter \(\varrho = 0.352\) in Plot (a) and \(\varrho = 0.936\) in Plot (b). Note that \(\cS(C_{\varrho}^{\text{Gauss}}) = C_{{\sqrt{1-\varrho^2}}}^{\text{Gauss}}\); see \eqref{eq:S_Gauss}.}
    \label{fig:S_Gauss}
\end{figure}

Proposition \ref{prop:repdW} and Proposition \ref{thm:main_trafo} yield the following representation of Wasserstein correlations in terms of rearranged copulas.

\begin{corollary}[Representation of \(\dW_c\) via rearranged copulas]\label{thm:RepWasserstein}~\\
    For the Wasserstein correlation in \eqref{def:WassersteinCorrelation}, we have
    \begin{align}\label{eqthm:RepWasserstein}
        \dW_c(Y,X)
        = \frac{\int_{[0,1]^2} c(u,u') \,\de \cS(C_{Y,X}^\uparrow)(u,u')}{\int_0^1 \int_0^1 c(u,u') \de u' \de u},
    \end{align}
    i.e., \(\dW_c(Y,X)\) is a functional of the reflected rearranged copula of \(C_{Y,X}\).
\end{corollary}

Another direct consequence of Proposition \ref{thm:main_trafo} is the following representation of rearranged dependence measures in terms of upper products.

\begin{corollary}[Representation of \(\sR_\mu\) via upper products]\label{cor:sr}~\\
The rearranged dependence measure in \eqref{eq:RearrDepMeasure} satisfies
\begin{align}\label{eq:rep_rearrdepmeasure}
    \sR_\mu(Y,X) = \mu(\cS(C_{Y,X}\vee \Pi)),
\end{align}
i.e., \(\sR_\mu\) is a functional of the reflected upper product of \(C_{Y,X}\) with \(\Pi\).
\end{corollary}

\section{Metric properties of the upper product transform \(\cT\)}\label{sec3}

In this section, we show that \(\cT\) coincides with \(\cS\) on the class of SI copulas, while \(\cT^2\) acts as a projection of a copula onto its increasing rearranged copula.

Therefore, recall Proposition \ref{thm:main_trafo}, according to which the reflection operator \(\cS\) permutes the upper product \(C\vee \Pi\) and the increasing rearranged copula of \(C\).
With the definition of \(\cT\) in \eqref{def:operatorT}, this can be written as
\begin{align}\label{eq:propS}
    \cS\circ \cT(C) = C^\uparrow \qquad\text{and} \qquad \cT(C) = \cS(C^\uparrow).
\end{align}
In the sequel, we are also interested in metric properties of \(\cT\). Therefore, we equip the class of bivariate copulas with the \(L^p\)-norm for copula derivatives (taken with respect to the second component). To be precise, we define
\begin{align}\label{defmetr}
    d_p(C,D) := \left( \int_0^1 \int_0^1 \lvert \partial_2 C(u,v) - \partial_2 D(u,v) \rvert^p \de u \de v\right)^{1/p} \quad \text{for } C,D\in \cC \text{ and } p\geq 1.
\end{align}
It is well-known that \(d_p\) defines a metric on \(\cC\); see \cite{Trutschnig-2011}.
For \(p=1\), we obtain isometry properties for \(\cS\) and \(\cT\) based on the following simple lemma.

\begin{lemma}\label{lemdel_2_norm}
    For right-continuous, decreasing functions \(f,g \colon [0,1]\to [0,1]\), we have
    \begin{align}\label{eq:lemdel_2_norm}
        \int_0^1 |f(x) - g(x)| \de x = \int_0^1 |f^{-1}(y) - g^{-1}(y) | \de y,
    \end{align}
    with the generalized inverse defined in \eqref{defgeninv}.
\end{lemma}

\begin{proof}
Since \(f(x) \leq y\) if and only if \(x \geq f^{-1}(y)\), and thus \(f(x) > y\) if and only if \(x < f^{-1}(y)\), we have
    \begin{align*}
        \int_0^1 |f(x) -g(x)|\de x &= \int_0^1 \int_0^1 \1_{\min\{f(x),g(x)\}\leq y < \max\{f(x),g(x)\}} \de y \de x \\
        &= \int_0^1 \int_0^1 \1_{\min\{f^{-1}(y),g^{-1}(y)\}\leq x < \max\{f^{-1}(y),g^{-1}(y)\}} \de x \de y \\
        &= \int_0^1 | f^{-1}(y) - g^{-1}(y)| \de y
    \end{align*}
    using Fubini's theorem.
\end{proof}

To establish several fundamental properties of \(\cT\), we use that the mapping \(\cS\) is an isometric reflection on the metric space \((\cC^\uparrow,d_1)\) as follows.

\begin{proposition}[\(\cS\) is a reflection]\label{the:refS}
    For all \(D,E\in \cC^\uparrow\), we have
\begin{enumerate}[label=(\roman*)]
\item \label{the:refS0} \(\cS(\cS(D)) = D\),
    \item \label{the:refS1} \(d_1(\cS(D) , \cS(E)) = d_1( D , E)\),
\end{enumerate}
\end{proposition}

\begin{proof}Statement \ref{the:refS0} follows from Proposition \ref{thm:main_trafo}.
For the proof of Statement \ref{the:refS1}, we have
\begin{align}
    \nonumber d_1(D,E) &= \int_0^1\int_0^1 \left|\partial_2^+ D(u,t) - \partial_2^+ E(u,t)\right| \de t \de u \\
    \label{eq_key}&= \int_0^1\int_0^1 \left| (\partial_2^+ D(u,\cdot))^{-1}(t) - (\partial_2^+ E(u,\cdot))^{-1}(t)\right| \de t \de u \\
    \nonumber &= \int_0^1\int_0^1 \left|\partial_2^+ \cS(D)(u,t) - \partial_2^+ \cS(E)(u,t)\right| \de t \de u = d_1(\cS(D),\cS(E))
\end{align}
for all \(u\in [0,1]\),
where we use Lemma \ref{lemdel_2_norm} for the second equality. The third equality follows from the definition of the reflection operator \(\cS\). \end{proof}

Recall Theorem \ref{prop13} according to which \(\cT^{2k}(D) = D^\uparrow\) and \(\cT^{2k-1}(D) = \cT(D)\) for all \(k\in \{1,2,3,\ldots\}\), where \(\cT^k\) denotes the \(k\)-fold composition of \(\cT\) with itself.
Due to the following result, the mapping \(\cT\) is a generalization of the reflection operator \(\cS\). In particular, it permutes increasing rearranged copulas and upper products with \(\Pi\). The reflection property of \(\cT\) on the class of SI copulas is visualized in Figure \ref{Fig2}.

\begin{corollary}[\(\cT = \cS\) on \(\cC^\uparrow\)]\label{cor_TeqS}
    We have
        \(\cT(D) = \cS(D)\) for all \(D\in \cC^\uparrow\).
\end{corollary}

\begin{proof}
Let \(D\in\cC^\uparrow\). Since \(D=D^\uparrow\), Proposition
\ref{thm:main_trafo}\,\ref{thm:main_trafo2} yields \(\cS(D)
    =
    \cS(D^\uparrow)
    =
    D\vee\Pi
    =
    \cT(D).
\)
\end{proof}

The following result establishes metric properties of the upper product transform \(\cT\). It shows that \(\cT^2\) acts as a contraction with respect to the metric \(d_p\). Further, for \(p=1\), the mapping \(\cT\) is an isometry on \(\cC^\uparrow\).

\begin{theorem}[Metric properties of \(\cT\)]\label{the_metric}
For all \(D,E\in \cC\), we have
    \begin{enumerate}[label=(\roman*)]
    \item \label{the:Tiso4c1} \(d_p(\cT^2(D), \cT^2(E)) \leq d_p( D,E)\) for all \(p\geq 1\),
        \item \label{the:Tiso4c}\label{the:Tiso1} \(d_1(\cT^k(D), \cT^k(E)) = d_1( \cT^2(D), \cT^2(E)) = d_1( D^\uparrow, E^\uparrow)\) for \(k\in \N\).
    \end{enumerate}
\end{theorem}

\begin{proof}
\ref{the:Tiso4c1}: Define for \(v\in [0,1]\) and \(U\sim \cU(0,1)\) the random variables \(V_v := \partial_2 D(v,U)\) and \(W_v := \partial_2 E(v,U)\). Further, define \(V_v^*:= \partial_2 D^\uparrow(v,U)\) and \(W_v^*:=\partial_2 E^\uparrow(v,U)\). By definition of the increasing rearranged copulas in \eqref{def:inc_rearranged_copula}, we have \(V_v\eqd V_v^*\), \(W_v\eqd W_v^*\), and that \(V_v^*\) and \(W_v^*\) are comonotone. For bounded random variables \(S\) and \(T\), we denote by \(S \leq_{cx} T\) the convex order defined  by \(\E \varphi(S) \leq \E \varphi(T)\) for all convex functions \(\varphi\colon \R\to \R\). Then we obtain
\begin{align*}
    V_v-W_v\geq_{cx} V_v^*-W_v^*;
\end{align*}
see e.g. \cite[Corollary 3.28\,(b)]{Ru-2013}. It follows that
\begin{align*}
    d_p^p(D,E) &= \int_0^1 \int_0^1 |\partial_2 D(v,t) - \partial_2 E(v,t) |^p \de t \de u  = \int_0^1 \E|V_v-W_v|^p \de v \\
    &\geq \int_0^1 \E|V_v^*-W_v^*|^p\de v
    = \int_0^1 \int_0^1 |\partial_2 D^\uparrow(v,t) - \partial_2 E^\uparrow(v,t) |^p \de t \de u = d_p^p(\cT^2(D),\cT^2(E)).
\end{align*}
\ref{the:Tiso4c}: Applying Proposition \ref{the:refS}\,\ref{the:refS1}, we obtain
\begin{align}\label{pro:the:Tiso4c}
    d_1(\cT(D), \cT(E)) = d_1(\cS(\cT(D)), \cS(\cT(E))) = d_1( \cT^2(D), \cT^2(E)) = d_1( D^\uparrow, E^\uparrow)
\end{align}
 where the second equality is due to Corollary \ref{cor_TeqS}. The last equality holds true by Theorem \ref{prop13}.
\end{proof}

\begin{remark}\label{rem:T}
    \begin{enumerate}[label=(\alph*)]
    \item Key for the isometry properties of \(\cS\) in Proposition \ref{the:refS}\,\ref{the:refS1} and \(\cT\) in Theorem \ref{the_metric}\,\ref{the:Tiso4c} is the integral identity for right-continuous, decreasing functions in \eqref{eq:lemdel_2_norm}. Similarly, for integrable distribution functions \(F,G\), it is well-known that
        \begin{align}\label{eq_form_gen_inv}
    \int_0^1 \left| F^{-1}(t) - G^{-1}(t) \right| \de t = \int_\R |F(x) - G(x) | \de x,
\end{align}
see e.g. \cite[Remark 2.19\,(iii)]{Villani-2003}. Note that the left-hand side of \eqref{eq_form_gen_inv} equals the \(1\)-Wasserstein distance in \eqref{def:OT} (with cost function \(c(x,y) = |y-x|\)) between \(\nu\) and \(\nu'\) with \(F_\nu = F\) and \(F_{\nu'}=G\).
    \item A central idea of our paper is the construction of the reflection operator \(\cS\) in \eqref{def:reflectop} by reflecting the copula derivative \(t\mapsto\partial_2^+ C(u,t)\) along the main diagonal. For \((Y,X)\sim C\), this corresponds to a reflection of the conditional distribution function \(t\mapsto F_{Y|X=t}(u)\) as a function in the \emph{conditioning variable} \(X\).
While the literature typically focuses on integral representations of the form \eqref{eq_form_gen_inv} to relate distances between quantile and distribution functions, our approach in \eqref{eq_key} is different.
Instead, we leverage the integral identity \eqref{eq:lemdel_2_norm} with respect to the conditioning variable, specifically exploiting the fact that for SI copulas \(C\), the partial derivatives \(t\mapsto \partial_2 C(u,t)\) are decreasing.
        \item \label{rem:T1} Due to Proposition~1.3 and Theorem~3.4(i), \(\mathcal T^2\) is an
idempotent nonexpansive retraction from \((\mathcal C,d_p)\) onto
\((\mathcal C^\uparrow,d_p)\) for every \(p\ge1\). Equivalently,
\(\mathcal T^2\) is a nonexpansive projection onto
\(\mathcal C^\uparrow\).
        \item
        By Theorem \ref{the_metric}\,\ref{eq:CCSI} and \eqref{def:intrearr}, \(\cT\) restricted to \((\cC^\uparrow,d_1)\) is an isometry, that is,
        \begin{align}
            d_1( \cT(D),\cT(E)) = d_1(D, E) \qquad \text{for all } D,E\in \cC^\uparrow.
        \end{align}
        Hence, \(\cT\) also shares the isometry property of \(\cS\) in Proposition \ref{the:refS}\,\ref{the:refS1}.
\item Due to the identities \(\cT(\Pi) = M\) and \(\cT(M) = \Pi\) (see Lemma \ref{lem:propuppprod}), fixed points of \(\cT\) are non-trivial. A more detailed study of fixed points is provided in Section \ref{sec4}.
    \end{enumerate}
\end{remark}

As a specific case of the following example, the upper product transform of Gaussian copulas is again Gaussian.

\begin{example}[Upper product of Gaussian copulas]\label{ex:uppprodGauss}
    The Gaussian copula with parameter \(\varrho\in [-1,1]\) is defined by
    \begin{align}\label{def:Gauss_Cop}
        C_\varrho^{\text{Gauss}}(v,u)
=
\Phi_\varrho\bigl(\Phi^{-1}(v), \Phi^{-1}(u)\bigr),
\qquad (v,u) \in [0,1]^2,
    \end{align}
where \(\Phi_\varrho\) is the distribution function of the zero-mean bivariate normal distribution with covariance matrix \(\Sigma = \left(\begin{smallmatrix}
    1 & \varrho \\ \varrho & 1
\end{smallmatrix}\right)\), and \(\Phi^{-1}\) denotes the standard normal quantile function. As a consequence of \cite[Theorem 2]{Ansari-2018b}, the upper product defined in \eqref{def:uppprod} of two Gaussian copulas is again Gaussian:
\begin{align}
    C_\varrho^{\text{Gauss}} \vee C_{\varrho_*}^{\text{Gauss}} = C_{m(\varrho,\varrho_*)}^{\text{Gauss}} \qquad \text{for } m(\varrho,\varrho_*):= \varrho\varrho_* + \sqrt{1-\varrho^2}\sqrt{1-\varrho_*^2}.
\end{align}
Setting \(\varrho_*=0\) implies \(C_{\varrho_*}^{\text{Gauss}} = \Pi\) and thus
\begin{align}
\label{eq:S_Gauss}\cS(C_\varrho^{\text{Gauss}}) &= C_{\sqrt{1-\varrho^2}}^{\text{Gauss}}\qquad \text{for } \varrho\in [0,1],\\
    \label{eq:T_Gauss} \cT(C_\varrho^{\text{Gauss}}) &= C_{\sqrt{1-\varrho^2}}^{\text{Gauss}} \qquad \text{for } \varrho\in [-1,1].
\end{align}
This confirms Proposition \ref{thm:main_trafo} and Theorem \ref{prop13} for Gaussian copulas. Note that the reflection operator \(\cS\) is defined for SI copulas and thus \(\varrho\) in \eqref{eq:S_Gauss} is restricted to \([0,1]\). In contrast, \(\cT\) is defined on the whole class \(\cC\), which allows to transform Gaussian copulas in \eqref{eq:T_Gauss} for all \(\varrho\in [-1,1]\). This example also verifies Corollary \ref{cor_TeqS} according to which \(\cT\) coincides with \(\cS\) on \(\cC^\uparrow\).
\end{example}

The identity in \eqref{eq:T_Gauss} motivates the study of further properties of upper products with \(\Pi\). First, we observe that the Gaussian copula with parameter \(\varrho = 1/\sqrt{2}\) is a fixed point of \(\cT\). This raises the question of whether additional fixed points of \(\cT\) exist. We address this question in Section \ref{sec4}.
Second, recall that the Gaussian copula family \((C_\varrho^{\text{Gauss}})_{\varrho\in [-1,1]}\) is \(\leq_{lo}\)-increasing in \(\varrho\); see e.g. \cite{Ansari-Rockel-2023}. Hence,
\begin{align}\label{eq_Gaussprop}
    |\varrho| \leq |\varrho'| \qquad \text{implies} \qquad \cT(C_\varrho^{\text{Gauss}}) \geq_{lo} \cT(C_{\varrho'}^{\text{Gauss}}).
\end{align}
It is therefore natural to ask under which conditions upper products are comparable with respect to the lower orthant order. We answer this question in Section \ref{sec5}.
Third, \(\cT(C_\varrho^{\text{Gauss}})\) is continuous in \(\varrho\) with respect to uniform convergence of copulas.
This implies by \eqref{cor:T_rep} continuity results for \(\dW_c\) and \(\sR_\mu\) in the Gaussian setting.
In Section \ref{sec6}, we investigate general continuity properties of \(\cT\) which yield continuity results for Wasserstein correlations and rearranged dependence measures.

\section{Fixed points of \(\cT\)}\label{sec4}

In this section, we give necessary and sufficient conditions for \(\cT(C) = C\). Therefore, let us denote by
\begin{align}\label{def:C_T}
    \cC_\cT := \{C\in \cC^\uparrow \mid \cT(C) = C\}
\end{align}
the set of copulas that are fixed points of \(\cT\).
Before we construct in Theorem \ref{lem:SI_cop_construction} elements of \(\cC_\cT\),
we first motivate why fixed points are interesting.

To this end, let us consider rearranged dependence measures generated by affine functionals \(\mu\) of the form \(C\mapsto\int f \de C\) where the function \(f\) is supermodular. Then, the rearranged dependence measure in \eqref{eq:RearrDepMeasure} admits the representation
    \begin{align}\label{eq:R_lineara}
        \sR_\mu(Y,X) =  \frac{\int f \de C_{Y,X}^\uparrow - \int f \de \Pi}{\int f \de M - \int f \de \Pi} =: \sR_f(Y,X),
    \end{align}
    whenever the integrals \(\int f \de M = \int_0^1 f(u,u) \de u\) and \(\int f \de \Pi = \int_0^1\int_0^1 f(u,v) \de u \de v\) exist and are distinct.
In particular, for the convex costs in \eqref{def:convexcost}, the function \(f = -c\) is supermodular and
    the expression in \eqref{eq:R_lineara} simplifies to
    \begin{align}\label{rem:fxp1}
        \sR_{-c}(Y,X) = \frac{-\int c \de C_{Y,X}^\uparrow + \int c \de \Pi}{\int c \de \Pi} = 1- \frac{\int c \de C_{Y,X}^\uparrow}{\int c \de \Pi}.
    \end{align}
    Now, if \(C_{Y,X}^\uparrow\) is a fixed point of \(\cT\), we obtain from the representation of the Wasserstein correlation in \eqref{eqthm:RepWasserstein} the identity
    \begin{align}\label{rem:fxp2}
        \dW_c(Y,X) = \frac{\int c \de C_{Y,X}^\uparrow}{\int c \de \Pi}.
    \end{align}
    Combining \eqref{rem:fxp1} and \eqref{rem:fxp2}, yields the following result.

\begin{proposition}\label{prop:motiv_fp}
    For any convex cost function \(c\) in \eqref{def:convexcost} and for all \(C\in \cC\) with \(C^\uparrow\in \cC_\cT\), we have
    \begin{align}
    \sR_{-c}(Y,X) + \dW_c(Y,X) = 1 \qquad \text{where } (Y,X)\sim C.
    \end{align}
\end{proposition}
Hence, on the set \(\cC_\cT\), the value of \(\sR_{-c}\) is completely determined by \(\dW_c\) and vice versa.
In the remainder of this section, we investigate the fixed point set \(\cC_\cT\) in more detail. Therefore, recall from Lemma \ref{lem:propuppprod}
that the upper product \(C\vee \Pi\) is SI. This implies
 \(\cC_\cT\subsetneq \cC^\uparrow\).
Moreover, by Proposition \ref{thm:main_trafo}, we know that \(C\vee \Pi = \cS(C)\). This yields the necessary and sufficient fixed point condition
\begin{align}\label{eq:fixpointEquation}
    C = \cS(C).
\end{align}
Writing \(h_v(t) := \partial_2^+ C(v,t)\), the identity in \eqref{eq:fixpointEquation} is equivalent to \begin{align}\label{eq:involution}
    h_v(t) = h_v^{-1}(t)
\end{align}
for all \(v\in [0,1]\) and for all \(t\in [0,1]\) outside a \(\lambda\)-null set which may depend on \(v\). In other words, for every \(v\in (0,1)\), the function \(h_v\) is self-inverse in the generalized-inverse sense.

In the following theorem,
we provide a construction of infinitely many copulas that satisfy the above identity. Therefore, we denote by
\(\cF_c\) the class of continuous and strictly increasing distribution functions on \(\R\). For i.i.d. random variables \(S\) and \(T\) with distribution function \(F\in \cF_c\), we write \(F\ast F =  F_{S+T}\) for the distribution function of the convolution.

\begin{theorem}[Construction of fixed points for \(\cT\)]\label{lem:SI_cop_construction}
    For \(F\in \cF_c\)\,, define \(G := F\ast F\). Then, for
    \begin{align}\label{def:h_u_conv}
        k_v(t) := F(G^{-1}(v)-F^{-1}(t))\quad \text{for all } (v,t)\in (0,1)^2,
\end{align}
the function \(C(v,u):= \int_0^u k_v(t) \de t\), \((v,u)\in [0,1]^2\), is an SI copula with \(\cT(C) = C\).
\end{theorem}

\begin{proof}    To show that \(C\) is a copula, we verify Lemma \ref{charSIcop}. Properties \ref{charSIcop1} and \ref{charSIcop2} are immediate.
    To prove \ref{charSIcop3}, let \(S,T\sim F\) be independent. Then we have
    \begin{align*}
        \int_0^1 k_v(t) \de t &= \int_0^1 F(G^{-1}(v) - F^{-1}(t)) \de t
        = \int_{-\infty}^\infty P(S\leq G^{-1}(v) - y) \de P^T(y) \\
        &= \int_{-\infty}^\infty P(S + T \leq G^{-1}(v) \mid T = y) \de P^T(y)
        = G(G^{-1}(v)) = v,
    \end{align*}
    where we substitute \(y = F^{-1}(t)\) for the second equality. For the third equality, we use independence of \(S\) and \(T\). The fourth equality follows from disintegration and \(G = F_{S+T}\). For the last equality, we use that \(F\) and, thus, \(G\) is continuous. Hence, \(C\) is a copula, which is SI because \(k_v\) is decreasing for all \(v\). It is straightforward to verify that \(k_v\) satisfies \eqref{eq:involution} for all \(v,t\in (0,1)\). Hence, we have \(\cT(C) = C\).
\end{proof}

We denote by
\(\cC_* := \{C\in \cC\mid \partial_2 C(v,t) = k_v(t) \,\forall v,t\in (0,1), \, k_v \text{ given by } \eqref{def:h_u_conv} \text{ for } F\in
\cF_c\}\)
the set of copulas constructed via convolutions in Theorem \ref{lem:SI_cop_construction}.
The following result summarizes the preceding considerations.

\begin{corollary}[Necessary and sufficient fixed point conditions]\label{theconv}
The following assertions hold true.
\begin{enumerate}[label=(\roman*)]
    \item \(\cC_\cT = \{C\in \cC \mid C = \cS(C)\}\)
    \item \(\cC_\cT = \{C\in \cC\mid h_v(t) = \partial_2 C(v,t) \text{ satisfies }\eqref{eq:involution} \text{ for all }v\in [0,1] \text{ and } \lambda\text{-almost all } t\in (0,1)\}\).
    \item \(\cC_*\subseteq \cC_\cT \subsetneq \cC^\uparrow\).
\end{enumerate}
\end{corollary}

In the following example, we verify Theorem \ref{lem:SI_cop_construction} for the standard normal distribution function.

\begin{example}
    For \(F = \Phi\), we have \(G(x) = (F\ast F)(x) = \Phi(x/\sqrt{2})\).
    This implies for \(\varrho = 1/\sqrt{2}\) that
    \begin{align}
        k_v(t) = \Phi(\sqrt{2} \Phi^{-1}(v) - \Phi^{-1}(t)) = \Phi\left(\frac{\Phi^{-1}(v) - \varrho \Phi^{-1}(t)}{\sqrt{1-\varrho^2}}\right) = \partial_2 C_\varrho^{\text{Gauss}}(v,t),
    \end{align}
    where the first equality is by definition of \(k_v\) in \eqref{def:h_u_conv}.
     Hence, \(C_{1/\sqrt 2}^{\text{Gauss}}\) is a fixed point of \(\cT\), which is in line with Equation \eqref{eq:T_Gauss}.
\end{example}

\section{A dependence order for \(\cT\)}\label{sec5}

To construct and analyze structural properties of dependence measures, a dependence order that ranks their values across statistical models is essential. Such a dependence order must range from independence to perfect functional dependence (recall Lemma \ref{prop_meta}).
The dependence order \(\preccurlyeq_1\) in Proposition \ref{cor_lc12} is quite restrictive because it acts only on the diagonal of the Markov product \(C\ast C\).
In contrast, the dependence orders \(\preccurlyeq_2\) and \(\preccurlyeq_3\) defined by a pointwise ordering of \(\cT(C)\) and \(\cT^2(C)\) in Corollary \ref{cor_C2_C3} generate a broader class of dependence measures.

In this section, we focus on pointwise ordering criteria on the copulas \(\cT(C)\) and \(\cT^2(C)\), respectively.
Therefore recall that, for bivariate copulas \(D\) and \(E\),
\begin{align}\label{eq_lo_supsub}
    D\leq_{lo} E \qquad \text{implies} \qquad \int f \de D \leq (\geq) \int f \de E
\end{align}
for all supermodular (respectively, submodular) functions \(f\) such that the expectations exist; see \cite[Theorem 2.5]{Muller-2000}.
Since the convex costs in \eqref{def:convexcost} are submodular functions, it follows from \eqref{cor:T_rep} that
\begin{align}\label{eq_consT}
    \cT(C_{Y,X}) \leq_{lo} \cT(C_{Y',X'}) \qquad \text{implies} \qquad \dW_c(Y,X) \geq \dW_c(Y',X').
\end{align}
For the following, we assume that \(\mu\) is increasing in the pointwise order on \(\cC^\uparrow\). Classical examples are concordance measures such as Kendall's tau, Spearman's rho, or the functionals in \eqref{eq:R_lineara}. For such \(\mu\), we obtain from \eqref{cor:T_rep} that
\begin{align}\label{eq_consT2}
    \cT^2(C_{Y,X}) \leq_{lo} \cT^2(C_{Y',X'}) \qquad \text{implies} \qquad \sR_\mu(Y,X) \leq \sR_\mu(Y',X').
\end{align}
To characterize a pointwise ordering of \(\cT\) and \(\cT^2\), we consider
the \emph{Schur order for copula derivatives} (with respect to the second component) defined for bivariate copulas \(D\) and \(E\) by
\begin{align}\label{def:partial2_S}
    D\leq_{\partial_2 S} E \qquad :\Longleftrightarrow \qquad \partial_2 D(v,\cdot) \prec_S \partial_2 E(v,\cdot) \quad \text{for all } v\in (0,1);
\end{align}
see \cite{Ansari-2021}.
Here, for measurable functions \(f,g\colon (0,1)\to [0,1]\), the Schur order \(f\prec_S g\) is defined via their decreasing rearrangements by \(\int_0^v f^*(t) \de t \leq \int_0^v g^*(t) \de t\) for all \(v\in (0,1)\) with equality for \(v = 1\) \cite[Section 3.2]{Ru-2013}.
Since constant functions are minimal elements in \(\prec_S\), the independence copula serves as the global minimal element for the \(\leq_{\partial_2 S}\)-order.
In contrast, among functions mapping into the interval \([0,1]\), maximal elements in Schur order take values in \(\{0,1\}\) \(\lambda\)-almost surely. This implies that perfect dependence copulas are maximal elements in the \(\leq_{\partial_2 S}\)-order.
Note that the \(\leq_{lo}\)-order and the \(\leq_{\partial_2 S}\)-order coincide on the class \(\cC^\uparrow\) of SI copulas; see \cite[Lemma 3.16\,(ii)]{Ansari-2021}.

The following theorem establishes the equivalence between the Schur order for copula derivatives in \eqref{def:partial2_S}, the pointwise order of upper products with \(\Pi\), and the pointwise order of rearranged copulas.

\begin{theorem}[Pointwise ordering of upper products with \(\Pi\)]\label{prop:ord_T}
Let \(D\) and \(E\) be bivariate copulas. Then, the following statements are equivalent.
    \begin{enumerate}[label=(\roman*)]
        \item \label{prop:ord_T1} \(D\leq_{\partial_2 S} E\),
        \item \label{prop:ord_T3} \(\cT(D) \geq_{lo} \cT(E)\),
        \item \label{prop:ord_T3b} \(\cT(D) \geq_{\partial_2 S} \cT(E)\),
        \item \label{prop:ord_T4} \(\cT^2(D) \leq_{lo} \cT^2(E)\),
        \item \label{prop:ord_T4b} \(\cT^2(D) \leq_{\partial_2 S} \cT^2(E)\).
    \end{enumerate}
\end{theorem}

\begin{proof}
For \(h_v^D(t) = \partial_2 D(v,t)\) and \(h_v^E(t) = \partial_2 E(v,t)\), we have \(h_v^D\prec_S h_v^E\) by definition of the Schur order.
By the Hardy-Littlewood-Polya theorem \cite[Theorem 3.21]{Ru-2013}, the latter is equivalent to \(h_v^D(U) \leq_{cx} h_v^E(U)\) for \(U\sim \cU(0,1)\).
By a characterization of the convex order via the stop-loss order, this again is equivalent to
\(\int_0^1 \min\{h_v(t),u\} \de t \geq \int_0^1 \min\{h_v^E(t),u\} \de t\) for all \(u\); see \cite[Theorems 1.5.3 and 1.5.7]{Muller-2002}. This proves the equivalence of \ref{prop:ord_T1} and \ref{prop:ord_T3}, and similarly, the equivalence of \ref{prop:ord_T3} and \ref{prop:ord_T4b}.
Since, by Theorem \ref{corcharT}\,\ref{corcharT1}, \(\cT(C)\) and \(\cT^2(C)\) are SI copulas and since for SI copulas the \(\leq_{lo}\)-order and the \(\leq_{\partial_2 S}\)-order coincide, \ref{prop:ord_T3} and \ref{prop:ord_T3b} as well as \ref{prop:ord_T4} and \ref{prop:ord_T4b} are equivalent.
\end{proof}

Recall the dependence orders \(\preccurlyeq_i\), \(i=1,2,3\), in Proposition \ref{cor_lc12} and Corollary \ref{cor_C2_C3}. These are implied by \(\leq_{\partial_2 S}\) as follows.

\begin{corollary}
    \(D\leq_{\partial_2 S} E\) implies \(D\preccurlyeq_i E\) for \(i=1,2,3\).
\end{corollary}

\begin{proof}
    For \(i=1\), the statement follows from \(D\ast D(u,u) = D^\uparrow\ast D^\uparrow (u,u) \leq E^\uparrow\ast E^\uparrow(u,u) = E\ast E(u,u)\), where the inequality is due to \cite[Theorem 3.7]{Ansari-2021}. For \(i=2,3\), the statement is given by Theorem \ref{prop:ord_T} noting that \(D\preccurlyeq_2 E\) is \(\cT(D) \geq_{lo} \cT(E)\) as well as \(D\preccurlyeq_3 E\) is \(\cT^2(D)\leq_{lo} \cT^2(E)\).
\end{proof}

\begin{remark}
    We refer to \cite{Ansari-Rockel-2023} for an overview of well-known parametric Archimedean, extreme-value, elliptical, and unclassified copula families that are \(\leq_{lo}\)- or \(\leq_{\partial_2 S}\)-increasing in their parameter. In particular, the Gaussian copula \(C_{\varrho}^{\text{Gauss}}\) is \(\leq_{\partial_2 S}\)-increasing in \(|\varrho|\). Hence, Theorem \ref{prop:ord_T} confirms the ordering in \eqref{eq_Gaussprop}.
\end{remark}

Combining Theorem \ref{prop:ord_T} with Eq. \eqref{eq_consT} and \eqref{eq_consT2}, we arrive at the following comparison result for Chatterjee's rank correlation, Wasserstein correlations, and rearranged dependence measures. The statement for \(\xi\) has already been established in \cite[Lemma 2.10]{Ansari-Rockel-2023}, using different notation.

\begin{corollary}\label{cor_54}
    For \(D,E\in \cC\), let \((Y,X)\sim D\) and \((Y',X')\sim E\). Then \(D\leq_{\partial_2 S} E\) implies
\begin{align*}
     \xi(Y,X)\leq \xi(Y',X'),\quad \dW_c(Y,X) \leq \dW_c(Y',X'), \quad\text{and}\quad \sR_\mu(Y,X)\leq \sR_\mu(Y',X').
\end{align*}
\end{corollary}

\begin{remark}
\begin{enumerate}[label=(\alph*)]
\item By Corollary \ref{cor_54}, the \(\leq_{\partial_2 S}\)-order underlies Chatterjee's rank correlation, Wasserstein correlations with convex cost, and rearranged dependence measures. Note that this order also underlies rank-based dependence measures of the type \eqref{eq_rep_chat_phi} and measures of sensitivity in \eqref{eq_rep_chat_phi2}; see \cite[Proposition 3.7]{Ansari-LFT-2023}.
    \item In Theorem \ref{prop:ord_T}, \(D\leq_{\partial_2 S} E\) is \emph{not} equivalent to \(D\leq_{lo} E\). As a counterexample, consider the lower Fr\'{e}chet copula \(W(u,v) = \max\{u+v-1,0\}\), which models countermonotonicity (i.e., \((Y,X)\sim W\) implies \(Y = 1-X\) almost surely). Since \(W\) is a perfect dependence copula, it is a maximal element with respect to \(\leq_{\partial_2 S}\), whereas \(W\) is the least element in the pointwise order of copulas. Hence, we have \(W\leq_{lo} \Pi\) but \(\Pi \leq_{\partial_2 S} W\), where both inequalities are strict.
    \item Upper products and Markov products exhibit fundamentally different behavior, as noted in Remark \ref{rem_prop_upp_prod}. This discrepancy also arises at the level of orderings. For example, let \(D\) and \(E\) be SI copulas with \(D\leq_{lo} E\). Then, we have \(D\vee \Pi \geq_{lo} E \vee \Pi\) by Theorem \ref{prop:ord_T}, but \(D\ast \Pi\leq_{lo} E\ast \Pi\); see \cite[Corollary 3.21]{Ansari-2021}. For the Markov product, we know that \(D\ast D \leq_{lo} E \ast E\) as a consequence of \cite[Theorem 3.7]{Ansari-2021}, whereas \(D\vee D = M = E\vee E\); see Lemma \ref{lem:propuppprod}\,\ref{lem:propuppprod2}.
\end{enumerate}
\end{remark}

\section{Continuity of \(\cT\)}\label{sec6}

While Markov products and upper products have fundamentally different ordering properties, they exhibit similar continuity properties as we study in this section.
More precisely, we are interested in a mode of convergence for \(C\) that ensures uniform convergence of the transforms \(\cT(C)\) and \(\cT^2(C)\) and thus implies continuity properties of Wasserstein correlations and rearranged dependence measures as a consequence of Theorem \ref{cor_TWR}.

The investigation of continuity properties for dependence measures is motivated by \cite{buecher2024} who show that any dependence measure \(\kappa\) satisfying Axioms \ref{prop2} and \ref{prop3} fails to be weakly continuous (that is, \((Y_n,X_n)\to (Y,X)\) weakly does not imply \(\kappa(Y_n,X_n)\to \kappa(Y,X)\)). This lack of continuity has several consequences: first, the empirical distribution function cannot serve as a basis for a consistent plug-in estimator. Second, these dependence measures lack robustness under slight model misspecifications within any metric of weak convergence. Furthermore, independence tests based on such measures may suffer from trivial power against sequences of alternatives converging weakly to independence.

A suitable mode of convergence for weak continuity of Markov products is \(\partial_2\)-convergence; see \cite[Section 3.3]{Ansari-Fuchs-2025}. For bivariate copulas \(D_n\) and \(D\), this is defined by
\begin{align}
    D_n \xrightarrow[]{~\partial_2 ~ } D \qquad :\Longleftrightarrow \qquad \int_0^1 \left|\partial_2 D_n(u,t) - \partial_2 D(u,t) \right| \de t \to 0 \quad \text{for all } u \in [0,1].
\end{align}
Note that \(D_n \xrightarrow{~\partial_2~} D\) is equivalent to \(d_1(D_n,D) \to 0\) where \(d_1\) is the metric induced by \eqref{defmetr} for \(p=1\). Both modes of convergence are also equivalent to the optimal-transport based concepts of convergence in Knothe-Rosenblatt distance and to convergence in adapted Wasserstein distance; see \cite{Ansari-Wiesel-2026} for details.

As the following result shows, also upper products are continuous with respect to \(\partial_2\)-convergence. Recall that \(\partial_2\)-convergence and uniform convergence are equivalent on the class of SI copulas; see \cite[Proposition 3.6]{Siburg-2021}.

\begin{theorem}[Continuity of \(\cT\)]\label{corcont}
    For bivariate copulas \(D_n,D\in \cC\), consider the following statements.
    \begin{enumerate}[label=(\roman*)]
        \item \label{corcont1} \(D_n \xrightarrow[]{~\partial_2 ~ } D\),
        \item \label{corcont2} \(\cT(D_n) \xrightarrow[]{~\partial_2 ~ } \cT(D)\),
        \item \label{corcont3} \(\cT(D_n) \to \cT(D)\) uniformly,
        \item \label{corcont4} \(\cT^2(D_n) \xrightarrow[]{~\partial_2 ~ } \cT^2(D)\),
        \item \label{corcont5} \(\cT^2(D_n) \to \cT^2(D)\) uniformly.
    \end{enumerate}
Then, \ref{corcont1} implies \ref{corcont2}, and \ref{corcont2}--\ref{corcont5} are equivalent.
\end{theorem}

\begin{proof}
As a consequence of the continuity result for general copula products in \cite[Theorem 2.23]{Ansari-2021}, Statement \ref{corcont1} implies \ref{corcont2}.
 The equivalence of \ref{corcont2} and \ref{corcont4} follows from the isometry in Theorem \ref{the_metric}\,\ref{the:Tiso4c}.
 The equivalence of \ref{corcont2} and \ref{corcont3} as well as the equivalence of \ref{corcont4} and \ref{corcont5} follows from the fact referenced above that for SI copulas, \(\partial_2\)-convergence and uniform convergence are equivalent.
\end{proof}

For the next result, we assume that \(\mu\) is continuous with respect to uniform convergence on the class of SI copulas. This is in particular the case for the standard choices of \(\mu\) already discussed.
Then the following corollary is a consequence of Theorem \ref{corcont} and the representation of Wasserstein correlations and rearranged dependence measures via \(\cT\) and \(\cT^2\) in \eqref{cor:T_rep}.
For a bivariate copula \(C\) with \((Y,X)\sim C\), we write \(\dW_c(C) := \dW_c(Y,X)\) and \(\sR_\mu(C) := \sR_\mu(Y,X)\).

\begin{corollary}
    If \(C_n \xrightarrow[]{\partial_2} C\), then \(\dW_c(C_n) \to \dW_c(C)\) and \(\sR_\mu(C_n) \to \sR_\mu(C)\).
\end{corollary}

The following example demonstrates that the upper product transform \(\cT\) is not weakly continuous, in the sense that \(C_n\to C\) uniformly does not imply \(\cT(C_n)\to \cT(C)\) uniformly.
This example is the analogue of \cite[Example 2]{Ansari-Fuchs-2025}, where it is shown that the Markov product likewise fails to be weakly continuous.

\begin{example}[Lack of weak continuity of \(\cT\)]
    The independence copula can be approximated uniformly by a sequence \((C_n)_{n\in \N}\) of so-called shuffle-of-min copulas; see \cite[Theorem 1]{Kimeldorf-1978}. This implies \(C_n\to \Pi\) uniformly and, for \((Y_n,X_n)\sim C_n\), \(Y_n\) perfectly depends on \(X_n\). Hence, \(C_n\) are perfect dependence copulas and thus \(\cT(C_n) = \Pi\) for all \(n\). Since \(\cT(\Pi) = M\), it follows that \(\cT\) is \emph{not} weakly continuous.
\end{example}

\section{Proofs of Section \ref{secIntro}}\label{sec_pr1}

\begin{proof}[Proof of Theorem \ref{corcharT}.]
\ref{corcharT1}: The first equality is a consequence of Proposition \ref{thm:main_trafo}. The second equality holds true due to \eqref{eq:CCSI}.\\
\ref{corcharT2} and \ref{corcharT3}:
    We have \(\Pi\leq_{\partial_2 S} D \leq_{\partial_2 S} C^*\) for all \(D\in \cC\) and for all perfect dependence copulas \(C^*\) where the inequalities are strict if and only if \(D\ne \Pi\) resp. \(D\) is not a perfect dependence copula. Further, we have \(\Pi\leq_{lo} E \leq_{lo} M\) for all \(E\in \cC^\uparrow\) with strict inequalities if and only if \(E\ne \Pi\) resp. \(E\ne M\). Then the statement follows from Theorem \ref{prop:ord_T}.
\end{proof}

\begin{proof}[Proof of Proposition \ref{prop:repdW}.]
Statement \ref{prop:repdW1} follows by applying the one-dimensional optimal transport identity in \eqref{def:OT} conditionally on \(X=x\) and then integrating with respect to \(x\).\\
 \ref{prop:repdW2}:
From disintegration we obtain
\begin{align*}
    P\big(F_{Y|X}^{-1}(U) \leq v, U\leq u\big)
    &= \int_0^1 P\big( F_{Y|X = t}^{-1}(U) \leq v, U\leq u ) \de t
    = \int_0^1 \min\{F_{Y|X = t}(v) , u\} \de t \\     &= \int_0^1 \min\{\partial_2 C_{Y,X}(v,t),u\} \de t =  C_{Y,X}\vee \Pi\, (v,u)
\end{align*}
where we use \(X\sim \cU(0,1)\) as well as independence of \(X\) and \(U\) for the first equality. The second equality follows with the identity \(F^{-1}(w)\leq x\) if and only if \(w\leq F(x)\) for any distribution function \(F\). For the third equality, we use that \(X,Y\sim \cU(0,1)\), which yields for all \(v\) that \(F_{Y|X=t}(v) = \partial_2 C(v,t)\) for \(\lambda\)-almost all \(t\in [0,1]\); see \cite[Equation (2.9.1)]{Nelsen-2006}.
\end{proof}

\begin{proof}[Proof of Theorem \ref{prop13}.]
    The statement follows from Corollary \ref{cor_TeqS} and Proposition \ref{thm:main_trafo}.
\end{proof}

\begin{proof}[Proof of Theorem \ref{cor_TWR}.]
    By \eqref{eqthm:RepWasserstein}, the Wasserstein correlation admits a representation via the reflection operator \(\cS\) through \(\cS(C_{Y,X}^\uparrow)\). Due to Corollary \ref{cor_TeqS}, \(\cS\) coincides with \(\cT\) on the class of SI copulas. Hence, the statement for the Wasserstein correlation follows from \(\cS(C_{Y,X}^\uparrow) = \cT(C_{Y,X}^\uparrow) = \cT(C_{Y,X})\) where the last equality holds true by Theorem \ref{prop13}. \\
    Due to \eqref{eq:rep_rearrdepmeasure}, the rearranged dependence measure admits a representation via \(\cS(C_{Y,X}\vee \Pi) \).
    Since \(C_{Y,X}\vee \Pi\) is SI (see Theorem \ref{corcharT}\,\ref{corcharT1}), we obtain from Proposition \ref{thm:main_trafo} and Theorem \ref{prop13} that \(\cS(C_{Y,X}\vee \Pi) = C_{Y,X}^\uparrow = \cT^2(C_{Y,X}) \). This proves the representation of the rearranged dependence measure.\\
    The characterizations of the values \(0\) and \(1\) via the upper product transform \(\cT\) follow from Theorem \ref{corcharT} together with the fact that Wasserstein correlations and rearranged dependence measures satisfy Axioms \ref{prop1}--\ref{prop3}.
\end{proof}

\begin{proof}[Proof of Lemma \ref{prop_meta}.]
Since \(\mu\) maps into \([0,1]\), property \ref{prop1} is trivial.

To prove characterization of independence, by surjectivity, there exists \(C_0\in\cC\) such that
\(\mu(C_0)=0\). By property \ref{prop_meta1}, we have \(\Pi\preccurlyeq C\) for all \(C\in\cC\).
In particular, \(\Pi\preccurlyeq C_0\), and monotonicity of \(\mu\) yields \(\mu(\Pi)\leq \mu(C_0)=0\). Since \(\mu\) takes values in \([0,1]\), it follows that \(\mu(\Pi)=0\).\\
Now, let \(C\neq\Pi\). We claim that \(C\not\preccurlyeq\Pi\).
Indeed, if \(C\preccurlyeq\Pi\), then, since \(\Pi\preccurlyeq D\) for every
\(D\in\cC\), transitivity would imply \(C\preccurlyeq D\) for all \(D\in\cC\).
By \ref{prop_meta1}, this would imply \(C=\Pi\), a contradiction.
Hence, \(\Pi\preccurlyeq C\) and \(C\not\preccurlyeq\Pi\).
Strict \(\preccurlyeq\)-monotonicity at independence therefore gives \(\mu(C)>\mu(\Pi)=0\)
Consequently, \(\mu(C)=0\) if and only if \(C=\Pi\), proves
Axiom \ref{prop2}.

To prove characterization of perfect dependence, by surjectivity, there exists \(C_1\in\cC\) such that \(\mu(C_1)=1\). Let \(E\) be a perfect dependence copula.
By \ref{prop_meta2}, \(D\preccurlyeq E\) for all \(D\in\cC\) and hence, in particular, \(C_1\preccurlyeq E\). Thus,
\(1=\mu(C_1)\leq\mu(E)\leq 1\) so that \(\mu(E)=1\).\\
Conversely, let \(C\) not be a perfect dependence copula and let
\(E\) be any perfect dependence copula. By \ref{prop_meta2},
\(C\preccurlyeq E\).
We claim that \(E\not\preccurlyeq C\). Indeed, if \(E\preccurlyeq C\), then,
since \(D\preccurlyeq E\) for every \(D\in\cC\), transitivity would yield
\(D\preccurlyeq C\) for all \(D\in\cC\).
By \ref{prop_meta2}, this would imply that \(C\) is a perfect
dependence copula, a contradiction. Hence, \(C\preccurlyeq E\) and \(E\not\preccurlyeq C\),
and strict \(\preccurlyeq\)-monotonicity at perfect dependence gives \(\mu(C)<\mu(E)=1\).
This proves Axiom \ref{prop3}.
\end{proof}

\begin{proof}[Proof of Proposition \ref{cor_C2_C3}]
    The statement follows from Theorem~\ref{corcharT} and the inequality in \eqref{eq_si_pqd}.
\end{proof}

\section{Conclusion}\label{sec7}

We have shown that the upper product, which has received relatively little attention in the literature, naturally underlies the construction of dependence measures. More specifically, the upper product transform \(\cT(C)=C\vee\Pi\) characterizes independence and perfect functional dependence. Moreover, it models conditional comonotonicity---a concept underlying Wasserstein correlations, rearranged dependence measures, measures of sensitivity, and related dependence measures. We have shown that upper products and increasing rearranged copula permute via a reflection of copula derivatives along the main diagonal. The rather counterintuitive behavior of upper products, illustrated in particular in Figure \ref{fig3}, motivates further investigation of their structural properties and may thereby contribute to a better understanding of the behavior of dependence measures.

\section*{Acknowledgement}

The first author was funded in whole by the Austrian Science Fund (FWF) {[10.55776/PAT1669224]} project \emph{Stochastic orders for functional dependence}.

\bibliographystyle{amsplain}
\bibliography{Upper_Products_arxiv}

\end{document}